\documentclass[12pt, oneside,reqno]{amsart}
\usepackage{amsmath, amsfonts, amssymb}
\usepackage{eucal}
\usepackage{mathrsfs}
\usepackage{latexsym}
\usepackage{bbm}
\usepackage{mathtools}
\usepackage{exscale}
\usepackage{cases}
\usepackage{xcolor}
\usepackage{comment}

\usepackage[pdfstartview=FitH,
            CJKbookmarks=true,
            bookmarksnumbered=true,
            bookmarksopen=true,
            colorlinks=true,
            linkcolor=blue,
            anchorcolor=blue,
            citecolor=red,
            urlcolor=blue
            ]{hyperref}

\numberwithin{equation}{section}
\newtheorem{theorem}{Theorem}[section]
\newtheorem{definition}[theorem]{Definition}
\newtheorem{lemma}[theorem]{Lemma}
\newtheorem{corollary}[theorem]{Corollary}
\newtheorem{proposition}[theorem]{Proposition}

\newtheorem{definition and theorem}[theorem]{Definition and Theorem}
\def\bl{\begin{lemma}}
\def\el{\end{lemma}}
\def\bc{\begin{corollary}}
\def\ec{\end{corollary}}
\def\bt{\begin{theorem}}
\def\et{\end{theorem}}
\def\bp{\begin{proposition}}
\def\ep{\end{proposition}}
\def\be{\begin{equation}}
\def\ee{\end{equation}}
\def\baa{\begin{align*}}
\def\eaa{\end{align*}}

\def\bd{\begin{definition}}
\def\ed{\end{definition}}

\allowdisplaybreaks[4]
\theoremstyle{plain}
\newtheorem{thm}{Theorem}[section]

\newtheorem*{ThmA}{Theorem A}
\newtheorem*{ThmB}{Theorem B}

\newtheorem*{ThmC}{Theorem C}

\newtheorem{lem}[thm]{Lemma}
\newtheorem{prop}[thm]{Proposition}
\newtheorem{defn}[thm]{Definition}
\newtheorem{cor}[thm]{Corollary}

\theoremstyle{remark}
\newtheorem{rem}[thm]{Remark}
\newcommand{\lsub}[1]{\hskip -1.0pt\lower.3ex\hbox{$_{#1}$}}

\theoremstyle{definition}

\newcommand{\sn}{\mathbb S^{n-1}}

\newcommand{\tr}{\mathbb R}
\newcommand{\rn}{\mathbb R^n}

\newcommand{\tH}{\mathcal{H}}

\newcommand{\di}{\int_{\rn}\int_{\rn}}

\newcommand{\rR}{{\rm R}}
\newcommand{\Ra}{{\rm R}_{\alpha}}
\newcommand{\Raf}{{\rm R}_{\alpha} f}
\newcommand{\Aff}{{\rm Aff}(n,1)}

\newcommand{\Ent}{{\rm E}}

\renewcommand{\chi}{\operatorname{1}}

\title[Chord Sobolev inequalities]{Chord Sobolev inequalities}
\author{Fernanda M. Ba$\hat{\rm E}$ta \quad  and  \quad Xiaxing Cai}

\begin{document}

\begin{abstract}
The paper establishes a new family of sharp analytic inequalities. Together with the fractional Sobolev inequalities of Almgren and Lieb \cite{AL}, they form a complete class of analytic inequalities, referred to as the chord Sobolev inequalities. A close connection between these inequalities and chord isoperimetric inequalities in integral geometry is established through a
functional extension of chord power integrals. The limiting cases of the chord Sobolev inequalities are derived, one of which yields a logarithmic Sobolev-type inequality. Combined with the work of Bourgain, Brezis, and Mironescu \cite{BBM}, these results complete the picture of the chord Sobolev inequalities, including
their endpoint cases.
\end{abstract}

\maketitle

\section{Introduction}

The {\it Sobolev inequality} states that for any $C^1$ function $f:\rn\rightarrow\tr$ with compact support,
\begin{equation}\label{Sobolev}
    \int_{\rn}|\nabla f(x)|dx\ge n\omega_n^{\frac{1}{n}}\|f\|_{\frac{n}{n-1}},
\end{equation}
where $|\nabla f(x)|$ is the Euclidean norm of the gradient of $f$ and $\|f\|_{p}=\big(\int_{\rn}|f(x)|^pdx\big)^{1/p}$ is the $L^p$ norm of $f$. Here $\omega_n=\pi^{n/2}/\Gamma(n/2+1)$ denotes the volume of the unit ball in $\rn$ and $\Gamma$ is the gamma function. This inequality lies at the interface of analysis and geometry, linking isoperimetric phenomena with functional inequalities, and has been the subject of extensive and influential research; see, e.g., \cite{Aub, BH, DEFFL, FN, FZ, HV, Ho, Yau, Zh4}. It extends to functions of bounded variation, serving as a functional analogue of the classical {\it isoperimetric inequality} (see, e.g., \cite{Aub2, BZ, Ci, Fed, FF}). For a measurable set $E\subset\rn$, the classical isoperimetric inequality states that
\[S(E)\ge n\omega_n^{\frac{1}{n}}|E|^{\frac{n-1}{n}},\]
where $S(E)$ and $|E|$ denote the surface area and volume of $E$, respectively.

Beyond the classical isoperimetric inequality, a number of important geometric inequalities  admit functional extensions. Notable examples include the Brunn–Minkowski inequality, the Blaschke–Santal\'o inequality, and the fractional isoperimetric inequality; their functional counterparts capture and naturally generalize the underlying geometric structures (see, e.g., \cite{AL, Ball,  FS,  Sch}). In particular, the fractional isoperimetric inequality can be lifted to the following fractional Sobolev inequality of Almgren and Lieb \cite{AL}; see also \cite{FS, Lud1, Lud2}. 
\vskip 5pt

\noindent{\bf Fractional Sobolev inequalities:} For $\alpha\in(-1,0)$ and $f\in W^{-\alpha,1}(\rn)$, there is a constant $\sigma_{n,\alpha}>0$ (with explicit value given in \eqref{optimal constant}; see also \cite{FFMMM, Garo}) such that
\begin{equation}\label{frac-Sob}
    2\sigma_{n,\alpha}\|f\|_{\frac{n}{n+\alpha}}\leq \di\frac{|f(x)-f(y)|}{|x-y|^{n-\alpha}}dxdy,
\end{equation}
where $W^{-\alpha,1}(\rn)$ denotes the fractional Sobolev space of integrable functions such that the right-hand side of \eqref{frac-Sob} is finite. Equality holds if and only if $f$ is a constant multiple of the characteristic function of a ball. 

Almgren and Lieb \cite{AL} established the P\'olya--Szeg\H{o} inequalities for 
%fractional Sobolev seminorms
%\[\|f\|_{W^{-\alpha,1}(\rn)}=\di\frac{|f(x)-f(y)|}{|x-y|^{n-\alpha}}dxdy.\]
the right-hand side of \eqref{frac-Sob} and the equality cases were later characterized by Frank and Seiringer \cite{FS}. A direct consequence is the above fractional Sobolev inequality. In particular, Bourgain, Brezis, and Mironescu \cite{BBM} showed that the fractional Sobolev inequality converges to the classical Sobolev inequality \eqref{Sobolev} in the limit $\alpha\rightarrow -1^+$. 

The main aim of this paper is to establish a new family of analytic inequalities which, together with the fractional Sobolev inequalities \eqref{frac-Sob}, form the {\it chord Sobolev inequalities} for $\alpha>-1$. These newly established inequalities, with their limiting cases, extend and complete the works of Almgren and Lieb \cite{AL} and Bourgain, Brezis and Mironescu \cite{BBM}.

    \begin{ThmA}\label{main1}
    For $\alpha\in (0,n)$ and non-negative $f\in L^{\frac{n}{n+\alpha}}(\rn)$, 
    \begin{equation}\label{chord-Sob}
        \sigma_{n,\alpha}\|f\|_{\frac{n}{n+\alpha}}\ge \di\frac{\min\{f(x),f(y)\}}{|x-y|^{n-\alpha}}dxdy,
    \end{equation}
    where the sharp constant is given by 
    \begin{equation}\label{optimal constant}
    \sigma_{n,\alpha}=\frac{n2^{\alpha}\pi^{\frac{n-\alpha-1}{2}}\Gamma(\frac{\alpha+1}{2})\Gamma(\frac{n}{2}+1)^{\frac{\alpha}{n}}}{|\alpha|\Gamma(\frac{n+\alpha}{2}+1)}.
\end{equation}
    Equality holds if and only if $f$ is a constant multiple of the characteristic function of a ball.
    \end{ThmA}

The inequality \eqref{chord-Sob} can be viewed as an $L^1$-analogue of the sharp Hardy--Littlewood--Sobolev inequality established by Lieb \cite{Lieb}. For $\alpha\in(0,n)$ and non-negative $f\in L^{\frac{2n}{n+\alpha}}(\rn)$, there is an explicitly known constant $\gamma_{n,\alpha}>0$ such that
\[\gamma_{n,\alpha}\|f\|_{\frac{2n}{n+\alpha}}^2\ge \di\frac{f(x)f(y)}{|x-y|^{n-\alpha}}dxdy,\]
which is equivalent to the fractional $L^2$ Sobolev inequality (see \cite{Ca2}).

In contrast to the \(L^2\) theory, the present \(L^1\) setting is naturally
connected to chord power integrals in integral geometry, as will be
discussed later. This connection motivates the choice of the kernel
\(\min\{f(x),f(y)\}\), which is 1-homogeneous in \(f\) and compatible with
the layer-cake decomposition.

The following theorem treats the case $\alpha>n$. Here $\|\cdot\|_\infty$ denotes the essential supremum norm.

\begin{ThmB}
    Let $\alpha>n$ and $f$ be a non-negative, non-zero function in $L^1(\mathbb{R}^n)\cap L^\infty(\mathbb{R}^n)$. Then 
\begin{equation}\label{reversed chord-Sob}
\sigma_{n,\alpha} \|f\|_1^{\frac{n+\alpha}{n}}\|f\|_\infty^{-\frac{\alpha}{n}}\leq \int_{\mathbb{R}^n }\int_{\mathbb{R}^n}
\frac{\min\{f(x),f(y)\}}{|x-y|^{n-\alpha}}dxdy,
\end{equation}
where $\sigma_{n,\alpha}$ is given in \eqref{optimal constant}. Equality holds if and only if $f$ is a constant multiple of the characteristic function of a ball. 
\end{ThmB}

The chord Sobolev inequalities are closely related to isoperimetric inequalities for random lines in integral geometry, known as {\it chord isoperimetric inequalities}. They can be interpreted as functional analogues of chord isoperimetric inequalities.

%The chord Sobolev inequalities arise from isoperimetric inequalities for random lines in integral geometry, known as {\it chord isoperimetric inequalities}. More precisely, they can be interpreted as functional analogues of isoperimetric inequalities for chord power integrals, with $\alpha$ ranging over different intervals.

For $\alpha>-1$ and a convex body $K\subset \rn$ (a convex compact set with non-empty interior), the chord power integral of $K$ is defined by
\[I_{\alpha+1}(K)=\int_{\Aff}|K\cap l|^{\alpha+1}dl,\]
where $|K\cap l|$ denotes the length of the intersection. Here $\Aff$ denotes the affine Grassmannian of $1$-dimensional lines in $\rn$ and $dl$ is the standard Haar measure on it (see Section \ref{chord and RaK} for details). The chord power integral has attracted increasing attention in recent years, mainly in connection with Minkowski-type problems (see, e.g., \cite{Cai2, GXZ, HHLW, LXYZ2020, XYZZ, XZ}). For recent results on inequalities involving chord power integrals, we refer the reader to \cite{XC, XS, Xu, Zh3}.

The chord isoperimetric inequalities state that there exists $\tilde{\sigma}_{n,\alpha}=\frac{|\alpha|(\alpha+1)}{n\omega_n}\sigma_{n,\alpha}$ such that
\begin{align}
    I_{\alpha+1}(K) &\ge \tilde{\sigma}_{n,\alpha}|K|^{\frac{n+\alpha}{n}}, \quad \alpha\in(-1,0)\cup (n,\infty), \label{iso-chord-011}\\
    I_{\alpha+1}(K) &\le \tilde{\sigma}_{n,\alpha}|K|^{\frac{n+\alpha}{n}}, \quad \alpha\in(0,n), \label{iso-chord-1n+11}
\end{align}
which were proved by Blaschke in $\tr^2$, and proved by  Ren \cite{Ren1}, by Davy \cite{Davy}, and by Schneider \cite{Sch2} in $\rn$ for integer $\alpha$. For non-integer $\alpha>-1$, these inequalities were proved by Zhang \cite{Zh2}; see also \cite{Ren2}. 

When $\alpha\rightarrow -1^+$, inequality \eqref{iso-chord-011} reduces to the classical isoperimetric inequality, which follows from Cauchy's integral formula, as discussed in Section \ref{chord and RaK}. Federer and Fleming \cite{FF} and Ma${\rm z}^{\prime}$ya \cite{Ma} independently discovered the remarkable connection between the sharp Sobolev inequality and the classical isoperimetric inequality. Zhang \cite{Zh4} proved a strengthening of this correspondence by establishing the affine Sobolev inequality and an integral affine isoperimetric inequality for domains. We next demonstrate the connection between the chord Sobolev inequalities and the chord isoperimetric inequalities.

The following integral representations of the chord power integrals play a central role in the development of the framework:
\begin{align}
    &I_{\alpha+1}(K)=\frac{|\alpha|(\alpha+1)}{2n\omega_n}\int_{\rn}\int_{\rn}\frac{|\chi_K(x)-\chi_K(y)|}{|x-y|^{n-\alpha}}dxdy, & & \alpha\in(-1,0), \label{chord-integral formula-01}\\ 
    &I_{\alpha+1}(K)=\frac{\alpha(\alpha+1)}{n\omega_n}\int_{\rn}\int_{\rn}\frac{\chi_K(x)\chi_K(y)}{|x-y|^{n-\alpha}}dxdy, & & \alpha>0, \label{chord-integral formula>1}
\end{align}
where $\chi_K$ denotes the characteristic function of $K$, taking value 1 in $K$ and 0 outside. Formula \eqref{chord-integral formula>1} is classical and related to Riesz potentials (see \cite{Ren2, San}). Formula \eqref{chord-integral formula-01} was first introduced by Qin \cite{Qin} (see \cite[Theorem 7.2.7]{SW} for a generalization) and is called the fractional integral of $K$. These formulas illustrate the geometric–functional correspondence between the chord power integral and the integrals appearing in \eqref{frac-Sob} and \eqref{chord-Sob}.

Motivated by these formulas, we introduce the following extension to non-negative measurable functions:
\begin{equation}\label{chord f-integral formula}
    \begin{aligned}
       &I_{\alpha+1}(f)=\frac{|\alpha|(\alpha+1)}{2n\omega_n}\int_{\rn}\int_{\rn}\frac{|f(x)-f(y)|}{|x-y|^{n-\alpha}}dxdy, &  & ~~\alpha\in(-1,0), \\
       &I_{\alpha+1}(f)=\frac{\alpha(\alpha+1)}{n\omega_n}\int_{\rn}\int_{\rn}\frac{ \min\{f(x), f(y)\} }{|x-y|^{n-\alpha}}dxdy,&  & ~~\alpha>0.
    \end{aligned}
\end{equation}
Here, for $\alpha>0$, we use $\min\{f(x), f(y)\}$ rather than $f(x)f(y)$ in order to preserve the $1$-homogeneity present in \eqref{chord-integral formula-01}. If $f$ is a log-concave function, \eqref{chord f-integral formula} admits a unified formulation in terms of level sets: 
\begin{equation}\label{fcn chord-level set}
     I_{\alpha+1}(f)=\int_0^{\infty}I_{\alpha+1}(\{f\ge t\})dt.
\end{equation}
In Section \ref{chord f}, we show that the integral in \eqref{fcn chord-level set} is finite and well-defined for $\alpha\ge -1$, provided that $f$ is log-concave and integrable.
%%%%%%%%%%%%%%%%%%%%%%%%%%%%%%%%%%%%%%%%%%%%%%%%%%%%%%%

Using the representation \eqref{fcn chord-level set}, for a log-concave, integrable function $f$, the chord Sobolev inequalities take the form
%Using the representation \eqref{fcn chord-level set}, for a log-concave, integrable function $f$, the chord Sobolev inequalities can be rewritten as
    \begin{align}
    I_{\alpha+1}(f)&\ge \tilde{\sigma}_{n,\alpha}\|f\|_{\frac{n}{n+\alpha}}, \quad\alpha\in(-1,0), \label{-10} \\
    I_{\alpha+1}(f)&\le \tilde{\sigma}_{n,\alpha}\|f\|_{\frac{n}{n+\alpha}}, \quad\alpha\in(0,n)\nonumber,
    \end{align}
and
\begin{equation*}
    I_{\alpha+1}(f)\ge \tilde{\sigma}_{n,\alpha}\|f\|_1^{\frac{n+\alpha}{n}}\|f\|_{\infty}^{-\frac{\alpha}{n}}, \quad \alpha>n,
\end{equation*}
where $\tilde{\sigma}_{n,\alpha}$ appears in \eqref{iso-chord-011} and \eqref{iso-chord-1n+11}. Inequality \eqref{-10} converges to the classical  Sobolev inequality by the work of Bourgain, Brezis and Mironescu \cite{BBM}. For characteristic functions $f=\chi_K$ with $K\subset \rn$ a convex body, these inequalities reduce to the chord isoperimetric inequalities \eqref{iso-chord-011} and \eqref{iso-chord-1n+11}. Note that in the functional setting, the inequalities for $\alpha>n$ take a different form from those for $\alpha\in(-1,0)$, unlike the chord isoperimetric inequalities, which have the same form in both ranges. These distinctions are discussed in Section \ref{Proof of Theorem A}.

%In Sections \ref{entropy} and \ref{log Sob ineq}, we investigate the limiting cases of the chord Sobolev inequalities at the endpoints $\alpha=0$ and $\alpha=n$.

In Sections \ref{entropy} and \ref{log Sob ineq}, we study the limiting behavior of the chord Sobolev inequalities at the endpoints $\alpha=0$ and $\alpha=n$. In the geometric setting, the chord isoperimetric inequalities \eqref{iso-chord-011} and \eqref{iso-chord-1n+11} reduce to identities in the limits $\alpha\to 0$ and $\alpha\to n$, respectively. Xu \cite{Xu} and Zhang \cite{Zh3} captured the first-order behavior at these endpoints through chord entropy inequalities. 

At the analytic level, the chord Sobolev inequality also collapses to an identity as $\alpha\to 0^+$, under mild integrability assumptions. From the other side, a result of Ma${\rm z}^{\prime}$ya and Shaposhnikova \cite{MS} shows that the fractional Sobolev inequality \eqref{frac-Sob} likewise reduces to an identity as $\alpha\to 0^-$, provided that $f\in W^{\beta,1}(\rn)$ for some $\beta\in(0,1)$. The endpoint $\alpha=0$ therefore gives rise to a functional chord entropy inequality.

In Section \ref{entropy}, using the level-set definition
\eqref{fcn chord-level set}, we extend the chord entropy inequalities to
log-concave functions in \(L^1(\rn)\). In Section \ref{log Sob ineq}, we
further extend the entropy inequality at \(\alpha=0\) beyond the log-concave
 setting to a fractional Sobolev class, obtaining the
logarithmic Sobolev-type inequality in Theorem C. Logarithmic Sobolev-type
inequalities have been extensively studied; see, e.g.,
\cite{BP, Cai, Ca, Le, Ro}.

%In Section \ref{entropy}, we employ the definition \eqref{fcn chord-level set} to extend the chord entropy inequalities to log-concave functions $f\in L^1(\rn)$. In Section \ref{log Sob ineq}, we extend this entropy inequality beyond the log-concave settings to a fractional Sobolev class, yielding the logarithmic Sobolev--type inequality in Theorem C. Logarithmic Sobolev--type inequalities have been extensively studied; see, e.g., \cite{BP, Cai, Ca, Le, Ro}.

\begin{ThmC}
    Let $\beta\in (0,1)$ be fixed, and $f\in W^{\beta,1}(\rn)$ be a non-negative function such that $\|f\|_1=1$ and $f\log f\in L^1(\rn)$. Then
    \begin{equation}\label{log}
    \sigma_0-\omega_n\int_{\rn}f(x)\log f(x)dx\ge \di\frac{\min\{f(x), f(y)\}-f(x)e^{-|x-y|} }{|x-y|^n}dxdy ,
    \end{equation}
    where the optimal constant is given by $\sigma_0=\frac{d}{d\alpha}|_{\alpha=0}\alpha\sigma_{n,\alpha}-n\omega_n\Gamma^{\prime}(1)$. Equality holds if and only if $f$ is a constant multiple of the characteristic function of a ball.
\end{ThmC}

The normalization $\|f\|_1=1$ is not essential in Theorem C; a general form is given in Section \ref{log Sob ineq}. Moreover, the right-hand side of \eqref{log} admits a geometric interpretation via dual mixed volumes. For further discussion, see Sections \ref{rdm-fcn}--\ref{log Sob ineq}.

\section{Preliminaries}

We collect some basic notation and results on $L^p$ spaces, Schwarz symmetrals, star-shaped sets. The books of Schneider \cite{Sch} and Lieb and Loss \cite{LL} are good general references.

\subsection{\texorpdfstring{$\boldsymbol{L^p}$}{} spaces}\label{Lp space}\hspace{\fill}

Let $p>-1$ and $(X,\mathcal{B},\mu)$ be a measure space. For $p\neq 0$ and a measurable function $f:X\rightarrow \tr$, let
\[\|f\|_p=\Big(\int_{X}|f(x)|^p d\mu(x)\Big)^{1/p}\]
and
\[\|f\|_{\infty}=\inf\{C>0: |f(x)|\leq C~~\text{almost everywhere on}~X\}.\]
In particular, if $p\in(-1,0)$ and $\mu(\{x\in X: f(x)=0\})>0$, we set $\|f\|_p=\infty$. The space $L^p(X)$ is then defined as
\[L^p(X)=\{f:X\rightarrow \tr : f~\text{is measurable},~\|f\|_p<\infty\},\]
where $f = g$ in $L^p(X)$ means that $f$ and $g$ are equal almost everywhere.

Assume that $\mu(X)$ is finite. For non-negative and $\mu$-integrable function $f$, the $p$-th mean of $f$ is defined by
\[M_pf=\Big(\frac{1}{\mu(X)}\int_{X}f(x)^pd\mu(x)\Big)^{1/p}\]
for $p\neq 0$. By Jensen's inequality,
\[M_pf\leq M_qf\]
for $-1<p<q<0$ or $0<p<q<\infty$. By the monotonicity of $M_pf$ with respect to $p$, there is
\[\lim_{p\rightarrow 0}\log M_pf=\frac{1}{\mu(X)}\int_X\log f(x)d\mu(x).\]

Similarly, for $p\neq0$, set
\[N_pf=\Big(\int_Xf(x)^{p+1}d\mu(x)\Big)^{1/p}.\]
If we additionally assume that $\mu(X)=1$, one can show that
\begin{equation}\label{ent f}
    \lim_{p\rightarrow 0}\log N_pf=\int_Xf(x)\log f(x)d\mu(x),
\end{equation}
which is the entropy of $f$. For more information about $p$-th mean, we refer the reader to \cite{GZ}.

\subsection{Symmetrization}\label{sym}\hspace{\fill}

Let $E\subset\rn$ be a Borel set of finite measure. The Schwarz symmetral of $E$, denoted by $E^{\star}$, is the closed ball centered at the origin with the same volume as $E$.

For a non-negative measurable function $f:\rn\rightarrow \tr$, the superlevel set of $f$ at height $t>0$ is defined by $\{f\ge t\}=\{x\in \rn: f(x)\ge t\}$. We say that $f$ is non-zero if $\{f\neq 0\}$ has positive measure. If the superlevel sets of $f$ have finite measure, the layer cake formula gives
  \[f(x)=\int_0^{\infty}\chi_{\{f\ge t\}}(x)dt\]
for almost all $x\in\rn$. 

The Schwarz symmetral of $f$, denoted by $f^{\star}$, is defined as
\[f^{\star}(x)=\int_0^{\infty}\chi_{\{f\ge t\}^{\star}}(x)dt\]
for $x\in\rn$. Hence $f^{\star}$ is radially symmetric and the superlevel set $\{f^{\star}\ge t\}$ has the same volume as $\{f\ge t\}$ for every $t> 0$. 
%The Schwarz symmetral $f^{\star}$ is also called the symmetric decreasing rearrangement of $f$. If $|x|<|y|$ implies $f^{\star}(x)>f^{\star}(y)$, we say $f^{\star}$ is strictly symmetric decreasing.

A key technique in our proofs  is the Riesz rearrangement inequality (see, e.g.,  \cite[Theorem 3.7]{LL}) and the characterization of equality cases due to Burchard  \cite{Bu}.

\begin{thm}[Riesz's rearrangement inequality.]\label{rri}
    For measurable $f,g,k:\rn\rightarrow [0, \infty)$ with superlevel sets of finite measure.
    \[\di f(x)k(x-y)g(y)dxdy\leq \di f^{\star}(x)k^{\star}(x-y)g^{\star}(y)dxdy.\]
\end{thm}

\begin{thm}[Burchard]\label{rri-cha}
    Let $A, B$ and $C$ be sets of finite positive measure in $\rn$ and denote by $a, b$ and $c$ the radii of their Schwarz symmetrals $A^{\star}, B^{\star}$ and $C^{\star}$. For $|a-b|<c< a+b$, there is equality in
    \[\di\chi_A(x)\chi_B(x-y)\chi_C(y)dxdy\leq \di\chi_{A^{\star}}(x)\chi_{B^{\star}}(x-y)\chi_{C^{\star}}(y)dxdy\]
    if and only if, up to sets of measure zero,
    {\[A=x_0+a D,~B=x_1+b D,~C=x_2+c D,\]}
    where $D$ is a centered ellipsoid, and $x_0,x_1$ and $x_2=x_0+x_1$ are vectors in $\rn$.
\end{thm}

\subsection{Star-shaped sets and dual mixed volumes.} \label{dmv}\hspace{\fill}

We say that $L\subset \rn$ is a star-shaped set (with respect to the origin) if for every $x\in L$ and $t\in[0,1]$, we have $tx\in L$. For a star-shaped set $L$, the radial function $\rho_L:\rn\setminus\{0\}\rightarrow [0,\infty]$ is defined by
\[\rho_{L}(x)=\sup\{\lambda\ge 0: \lambda x\in L\}.\]
A star-shaped set $L$ is said to be a star body if $\rho_{L}$ is continuous and strictly positive on $\rn\setminus\{0\}$. We say that two star-shaped sets are dilates if  $\rho_K(u)=c\rho_L(u)$ for some $c>0$ and almost all $u\in\sn$.

For a star-shaped set $L$ with measurable radial function, the $n$-dimensional volume, or the $n$-dimensional Lebesgue measure of $L$ is given by
\[|L|=\frac{1}{n}\int_{\sn}\rho_{L}(u)^ndu,\]
where $du$ is the spherical Lebesgue measure. The dual mixed volume for star bodies $K$ and $L$ introduced by Lutwak  \cite{Lut1} is given by
\[\tilde{V}_{\alpha}(K,L)=\frac{1}{n}\int_{\sn}\rho_{K}(u)^{n-\alpha}\rho_L(u)^{\alpha}du,\]
for $\alpha\in\tr\setminus\{0,n\}$. When $L=B^n$, the quantity $\tilde{V}_\alpha(K, B^n)$ is also known as the dual volume of $K$.  In particular, 
\[\tilde{V}_{\log}(K,L)=\frac{1}{n|K|}\int_{\sn}\rho_K(u)^n\log\Big(\frac{\rho_L(u)}{\rho_K(u)}\Big)du,\]
and this quantity arises as the limit 
\[\tilde{V}_{\log}(K,L)=\lim_{\alpha\rightarrow 0}\frac{1}{\alpha}\log\Big(\frac{\tilde{V}_{\alpha}(K,L)}{|K|}\Big).\]
For simplicity, we formulate the above definitions for star bodies. 
For general star-shaped sets, the definitions extend by restricting integration to directions where the radial functions are positive.

\section{Chord Sobolev inequalities}\label{Proof of Theorem A}

In this section, we focus on the proof of Theorem A and Theorem B. Before presenting the proof, we first establish a technical lemma, which is based on the Minkowski inequality (see, for example, \cite[Theorem 2.4]{LL}). Let $F$ be a non-negative, measurable function on $(X\times Y, \mu\times\nu)$ and $p\ge 1$. The Minkowski inequality states that
\begin{equation}\label{Mink-ineq}
    \int_{Y}\Big(\int_{X}F(x,y)^pd\mu(x)\Big)^{1/p}d\nu(y) \ge \Big(\int_{X}\Big(\int_{Y}F(x, y)d\nu(y)\Big)^{p}d\mu(x)\Big)^{1/p}.
\end{equation}

Equality and finiteness in \eqref{Mink-ineq} for $p\in(1,\infty)$ imply the existence of measurable functions $\varphi:X\rightarrow [0,\infty)$ and $\psi:Y\rightarrow [0,\infty)$ such that $F(x,y)=\varphi(x)\psi(y)$ almost everywhere.

We first consider the case $\alpha \in (0,n)$. The following lemma is a key step in the proof of Theorem A. Related arguments for the cases $\alpha \in (-1,0)$ and $\alpha = -1$ were established by Ludwig \cite{Lud1} and Zhang \cite{Zh4}, respectively; see also \cite{HL1} for further discussion.

\begin{lem}\label{m1}
    For $\alpha\in (0,n)$ and non-negative $f\in L^{\frac{n}{n+\alpha}}(\rn)$, 
    \[\|f\|_{\frac{n}{n+\alpha}}\ge \int_0^{\infty}|\{f\ge t\}|^{\frac{n+\alpha}{n}}dt.\]
    Equality holds if and only if $f=c\chi_E$ for some measurable set $E$ of finite measure and $c>0$ is a constant.
\end{lem}

\begin{proof}
    By the Minkowski inequality \eqref{Mink-ineq},
    \begin{equation}\label{Mink-int}
        \begin{aligned}
            \int_{\rn}f(x)^{\frac{n}{n+\alpha}}dx&= \int_{\rn}\Big(\int_{0}^{\infty}\chi_{\{f\ge t\}}(x)^{\frac{n+\alpha}{n}}dt\Big)^{\frac{n}{n+\alpha}}dx\\
            &\ge \Big(\int_0^{\infty}\Big(\int_{\rn}\chi_{\{f\ge t\}}(x)dx\Big)^{\frac{n+\alpha}{n}}dt\Big)^{\frac{n}{n+\alpha}}\\&=\Big(\int_0^{\infty}|\{f\ge t\}|^{\frac{n+\alpha}{n}}dt\Big)^{\frac{n}{n+\alpha}},
        \end{aligned}
    \end{equation}
    that is
    \[\|f\|_{\frac{n}{n+\alpha}}\ge \int_0^{\infty}|\{f\ge t\}|^{\frac{n+\alpha}{n}}dt.\]

    It is clear that equality is attained when $f=c\chi_E$. Conversely, if equality holds in \eqref{Mink-int}, there are non-negative measurable functions $\varphi$ and $\psi$ such that
    \[\chi_{\{f\ge t\}}(x)=\varphi(x)\psi(t).\]
    This implies that
    \[\{(x,t)\in\rn\times\tr: f(x)\ge t\}=\{x\in\rn: \varphi(x) > 0\}\times\{t\in\tr: \psi(t)> 0\}.\]
    Therefore, for every fixed $t>0$ such that $t\leq f(x)$ for some $x$, the superlevel set $\{f\ge t\}$ is a set independent of $t$, which implies that  $f=c\chi_E$ for some measurable set $E$ of finite measure and $c>0$ is a constant. 
\end{proof}

We are now ready to prove Theorem A. 

\begin{proof}[Proof of Theorem A]
    Note that
    \[\min\{f(x),f(y)\}=\int_0^{\infty}\chi_{\{f\ge t\}}(x)\chi_{\{f\ge t\}}(y)dt,\]
    and 
    \[|z|^{\alpha-n}=\int_0^{\infty}\chi_{r^{-\frac{1}{n-\alpha}}B^n }(z)dr,\qquad z\in\rn\backslash\{o\} .\]

    Then we have
    \begin{equation}\label{rep-A}
        \begin{aligned}
            \di\frac{\min\{f(x), f(y)\} }{|x-y|^{n-\alpha}}dxdy
            &=\int_0^{\infty}\int_{\rn}\int_{\rn}\frac{\chi_{\{f\ge t\}}(x)\chi_{\{f\ge t\}}(y)}{|x-y|^{n-\alpha}}dxdydt\\
            &=\int_0^{\infty}\int_0^{\infty}\di\chi_{\{f\ge t\} }(x)\chi_{r^{\frac{1}{\alpha-n}}B^n}(x-y)\chi_{\{f\ge t\}}(y)dxdydrdt.
        \end{aligned}
    \end{equation}

    The Riesz rearrangement inequality, Theorem \ref{rri-cha}, yields that
    \begin{equation}\label{rri-A}
    \begin{aligned}
        \di\chi_{\{f\ge t\} }(x)\chi_{r^\frac{1}{\alpha-n}B^n}(x-y)&\chi_{\{f\ge t\} }(y)dxdy\\
        &\leq \di\chi_{\{f\ge t\}^\star}(x)\chi_{r^\frac{1}{\alpha-n}B^n}(x-y)\chi_{\{f\ge t\}^\star}(y)dxdy.
    \end{aligned}
    \end{equation}
    Together with \eqref{rep-A}, we have
    \begin{equation}\label{rri-step-0n}
        \begin{aligned}
            \di\frac{\min\{f(x), f(y)\} }{|x-y|^{n-\alpha}}dxdy
            &\leq\int_0^\infty\di \frac{\chi_{\{f\ge t\}^\star }(x)\chi_{\{f\ge t\}\star }(y) }{|x-y|^{n-\alpha}}dxdydt\\&=\sigma_{n,\alpha}\int_0^{\infty}|\{f\ge t\}|^\frac{n+\alpha}{n}dt,
        \end{aligned}
    \end{equation}
    where the last equality follows from \eqref{chord-integral formula>1} applied to balls and 
$|\{f\ge t\}^\star|=|\{f\ge t\}|$. Then Lemma \ref{m1} implies that
    \begin{equation}\label{main1-step}
        \sigma_{n,\alpha}\|f\|_{\frac{n}{n+\alpha}}\ge \int_{\rn}\int_{\rn}\frac{\min\{f(x),f(y)\}}{|x-y|^{n-\alpha}}dxdy.
    \end{equation}

    When equality holds, it follows from Lemma \ref{m1} that $f=c\chi_E$ for some measurable set $E\subset \rn$. Moreover, equality must also occur in the rearrangement inequality \eqref{rri-A}, that is
    \[\di \chi_E(x)\chi_{r^{\frac{1}{\alpha-n}}B^n}(x-y)\chi_E(y)dxdy=\di \chi_{E^\star}(x)\chi_{r^{\frac{1}{\alpha-n}}B^n}(x-y)\chi_{E^\star}(y)dxdy\]
    for almost all $r>0$. The assumptions in Theorem \ref{rri-cha} are fulfilled when $r>0$ is sufficiently large. Consequently, $E$ is a ball, which concludes the proof.
\end{proof}

In the setting of $\alpha>n$, Theorem A does not admit a straightforward reversal, which will be implied by the following result and Lemma \ref{Mink-ineq}.

\begin{lem}\label{q>n}
    Let $q>0$ and  $f:\rn\rightarrow [0,\infty)$ be an integrable function such that
    \[\di\min\{f(x), f(y)\}|x-y|^q dxdy<\infty.\]
    Then
    \[\di\min\{f(x), f(y)\}|x-y|^qdxdy \geq \di\min\{f^{\star}(x), f^{\star}(y)\}|x-y|^qdxdy.\]
    Equality holds if and only if the superlevel sets $\{f\ge t\}$ are balls for almost all $t>0$, up to null sets.
\end{lem}

\begin{proof}
    For $z\in\mathbb{R}^n\setminus\{o\}$, 
\begin{align*}
    |z|^q=\int_0^\infty \chi_{\mathbb{R}^n\setminus r^{\frac{1}{q}}B^n}(z)dr.
\end{align*}
Using  Fubini's theorem, we obtain
\begin{equation}\label{cal-adi}
\begin{aligned}
\int_{\mathbb{R}^n}  \int_{\mathbb{R}^n} \min\{f(x),~&f(y)\}|x-y|^qdxdy\\
&= \int_0^\infty  \int_{\rn}  \int_{\rn} \chi_{\{f\ge t\}}(x)|x-y|^q \chi_{\{f\ge t\}}(y)dxdydt\\
 &= \int_0^\infty  \int_0^\infty \int_{\rn}  \int_{\rn}\chi_{\{f\ge t\}}(x)\chi_{\mathbb{R}^n\setminus r^{\frac{1}{q}}B^n}(x-y) \chi_{\{f\ge t\}}(y)dxdydrdt\\
 &=\int_0^\infty  \int_0^\infty \int_{\rn}  \int_{\rn}\chi_{\{f\ge t\}}(x)(1-\chi_{r^{\frac{1}{q}}B^n}(x-y)) \chi_{\{f\ge t\}}(y)dxdydrdt.
\end{aligned}
\end{equation}

The Riesz rearrangement inequality, Theorem \ref{rri-cha}, implies that
\begin{equation}\label{rri-chi-step}
    \begin{aligned}
    \di \chi_{\{f\ge t\}}(x)\chi_{r^{\frac{1}{q}}B^n}(x-y)\chi&_{\{f\ge t\}}(y)dxdy\\
    &\le\di \chi_{\{f\ge t\}^{\star}}(x)\chi_{r^{\frac{1}{q}}B^{n}}(x-y)\chi_{\{f\ge t\}^{\star}}(y)dxdy.
\end{aligned}
\end{equation}
Note that  \(\int_{\rn}\chi_{\{f\ge t\}}(x)dx<\infty\) for almost all $t>0$, as $f\in L^1(\rn)$. Combining \eqref{rri-chi-step} with the calculation in \eqref{cal-adi}, we have
\begin{align}\label{ri}
 \int_{\mathbb{R}^n}  \int_{\mathbb{R}^n} \min\{f(x),f(y)\}|x-y|^qdxdy \geq \int_{\mathbb{R}^n}  \int_{\mathbb{R}^n} \min\{f^\star(x),f^\star(y)\}|x-y|^qdxdy.   
\end{align}

Moreover, if equality holds in \eqref{ri}, for almost all $t>0$, equality holds in \eqref{rri-chi-step} for almost all $r>0$. Fix such a $t>0$ with $|\{f\ge t\}|>0$, we choose $r>0$ sufficiently small so that the assumptions of Theorem 2.2 are satisfied. Therefore $\{f\ge t\}$ is a ball, up to a null set.
\end{proof}

With the previous lemma established, we now consider the case $\alpha > n$, for which we have
\begin{equation}\label{reversed chord}
    \begin{aligned}
        \di\frac{\min\{f(x), f(y)\}}{|x-y|^{n-\alpha}}dxdy&\ge \di\frac{\min\{f^{\star}(x), f^{\star}(y)\}}{|x-y|^{n-\alpha}}dxdy\\
        &=\int_0^{\infty}\di \frac{\chi_{\{f^{\star}\ge t\}}(x)\chi_{\{f^{\star}\ge t\}}(y)}{|x-y|^{n-\alpha}}dxdydt\\
        &=\sigma_{n,\alpha}\int_0^{\infty}|\{f^{\star}\ge t\}|^{\frac{n+\alpha}{n}}dt\\
        &=\sigma_{n,\alpha}\int_0^{\infty}|\{f\ge t\}|^{\frac{n+\alpha}{n}}dt.
    \end{aligned}
\end{equation}

However, by Lemma \ref{m1} and the fact that $\frac{n+\alpha}{n}>1$, we only obtain
\[
\sigma_{n,\alpha}\int_0^{\infty}|\{f\ge t\}|^{\frac{n+\alpha}{n}}\,dt
\le \sigma_{n,\alpha}\|f\|_{\frac{n}{n+\alpha}},
\]
which shows that a direct reverse of \eqref{main1-step} is not available for $\alpha>n$. For this reason, we restrict our attention to functions $f\in L^1(\rn)\cap L^{\infty}(\rn)$ in the case $\alpha>n$.

\begin{proof}[Proof of Theorem B]
If the right-hand side of \eqref{reversed chord-Sob} is infinite, the inequality holds trivially. We may therefore assume that it is finite.
    As we calculated in \eqref{reversed chord},
    \[\di\frac{\min\{f(x), f(y)\}}{|x-y|^{n-\alpha}}dxdy\ge \sigma_{n,\alpha}\int_0^{\|f\|_{\infty}}|\{f\ge t\}|^{\frac{n+\alpha}{n}}dt.\]
    By Jensen's inequality, we get
    \begin{equation}\label{Jensen}
        \frac{1}{\|f\|_{\infty}}\int_0^{\|f\|_{\infty}}|\{f\ge t\}|^{\frac{n+\alpha}{n}}dt\ge \Big(\frac{1}{\|f\|_{\infty}}\int_0^{\|f\|_{\infty}}|\{f\ge t\}|dt\Big)^{\frac{n+\alpha}{n}}=\|f\|_{\infty}^{-\frac{n+\alpha}{n}}\|f\|_1^{\frac{n+\alpha}{n}},
    \end{equation}
    and thus
    \begin{equation}\label{ineq}
        \sigma_{n,\alpha}\|f\|_1^{\frac{n+\alpha}{n}}\|f\|_{\infty}^{-\frac{\alpha}{n}}\leq \di\frac{\min\{f(x), f(y)\}}{|x-y|^{n-\alpha}}dxdy.
    \end{equation}

    If equality holds in \eqref{ineq}, then equality must occur in the application of Jensen's inequality in \eqref{Jensen}. This in turn implies that  $|\{f\ge t\}|$ is constant for almost all $t\in (0,\|f\|_{\infty})$. Equivalently, this happens if and only if $f=c\chi_E$, where $E$ is a measurable set of finite measure and $c>0$. 

    Moreover, equality holds in \eqref{reversed chord} as well. In this case, Lemma \ref{q>n} indicates that $E$ is a ball, which  completes the characterization of the equality cases. 
\end{proof}

\section{Chord power integrals and radial mean bodies}\label{chord and RaK}

In this section, we recall some basic properties of chord power integrals and radial mean bodies, which are needed for the extension from convex bodies to log-concave functions in Sections \ref{chord f} and \ref{rdm-fcn}.

For a convex body $K\subset \rn$ and $\alpha\ge -1$, the chord power integral $I_{\alpha+1}(K)$ is defined by
\[I_{\alpha+1}(K)=\int_{\Aff}|K\cap l|^{\alpha+1}dl,\]
where $|K\cap l|$ denotes the length of the intersection. %Recall that $\Aff$ denotes the affine Grassmannian consisting of all $1$-dimensional lines in $\rn$ and $dl$ is the unique Haar measure on it.
Note that a random line $l\in \Aff$ is uniquely determined by an antipodal pair $u, -u\in\sn$ and the intersection point $x\in u^{\perp}\cap l$, where $u^{\perp}=\{y\in\rn: y\cdot u=0\}$. That is, $l=x+\tr u$. The Haar measure $dl$ is normalized so that
    \[dl=\frac{1}{n\omega_n}d\tH^{n-1}(x)du,\]
    where $du$ is the spherical Lebesgue measure on $\sn$. We denote by $\tH^{n-1}$ the $(n-1)$-dimensional Hausdorff measure on $u^{\perp}$, which coincides with the $(n-1)$-dimensional Lebesgue measure on $u^{\perp}$. The chord power integral can then be rewritten as
    \begin{equation}\label{chord-Xray}
        I_{\alpha+1}(K)=\frac{1}{n\omega_n}\int_{\sn}\int_{K|u^{\perp}}X_K(x,u)^{\alpha+1}d\tH^{n-1}(x)du,
    \end{equation}
    where $K|u^{\perp}$ denotes the orthogonal projection of $K$ onto $u^{\perp}$ and $X_K(x,u)=|K\cap (x+\tr u)|$ is known as the X-ray function of $K$.

There are three fundamental integral formulas connecting chord power integrals to the volume and surface area:
\begin{equation}\label{CCPH}
    I_0(K)=\frac{\omega_{n-1}}{n\omega_n}S(K),~~ I_1(K)=|K|,~~I_{n+1}(K)=\frac{n+1}{\omega_n}|K|^2,
\end{equation}
which are known, respectively, as Cauchy's integral formula, Crofton's volume formula, and Poincar\'e-Hadwiger's integral formula.

Recall that the chord isoperimetric inequality states that
\begin{align}
    I_{\alpha+1}(K) &\ge \tilde{\sigma}_{n,\alpha}|K|^{\frac{n+\alpha}{n}}, \quad \alpha\in(-1,0)\cup (n,\infty), \label{iso-chord-01}\\
    I_{\alpha+1}(K) &\le \tilde{\sigma}_{n,\alpha}|K|^{\frac{n+\alpha}{n}}, \quad \alpha\in(0,n), \label{iso-chord-1n+1}
\end{align}
where $\tilde{\sigma}_{n,\alpha}$ is given by
\begin{equation}\label{tilde sigma}
    \tilde{\sigma}_{n,\alpha}=\frac{|\alpha|(\alpha+1)}{n\omega_n}\sigma_{n,\alpha}=I_{\alpha+1}(B^n)\omega_n^{-\frac{n+\alpha}{n}}.
\end{equation}
By the Cauchy's integral formula in \eqref{CCPH}, the chord isoperimetric inequality converges to the classical isoperimetric inequality as $\alpha\rightarrow -1^+$.

We next relate chord power integrals to radial mean bodies, which were defined by Gardner and Zhang \cite{GZ}. For a convex body $K\subset\rn$ and $\alpha>-1$, the $\alpha$-th radial mean body $\Ra K$ is defined by
\begin{equation}\label{RaK}
    \rho_{\Ra K}(u)^{\alpha}=\frac{1}{|K|}\int_K\rho_{K-x}(u)^{\alpha}dx, ~\alpha\neq 0,
\end{equation}
and
\begin{equation}\label{R0K}
    \log\rho_{\rR_0 K}(u)=\frac{1}{|K|}\int_{K}\log\rho_{K-x}(u)dx.
\end{equation}
The following basic result shows the relation between $I_{\alpha+1}(K)$ and $\Ra K$.

\begin{lem}\label{chord-RaK}
    Let $\alpha>-1$. For a convex body $K\subset\rn$,
    \begin{equation}\label{dmv-RaK}
        I_{\alpha+1}(K)=\frac{(\alpha+1)|K|}{n\omega_n}\int_{\sn}\rho_{\Ra K}(u)^{\alpha}du.
    \end{equation}
\end{lem}

\begin{proof}
    If $\alpha=0$, \eqref{dmv-RaK} follows from the Crofton's volume formula in \eqref{CCPH} directly. We assume that $\alpha>-1$ and $\alpha\neq 0$. By the definition \eqref{RaK},
    \begin{equation*}
            \rho_{\Ra K}(u)^{\alpha}=\frac{1}{|K|}\int_K\rho_{K-x}(u)^{\alpha}dx=\frac{1}{|K|}\int_{K|u^{\perp}}\int_{s_1}^{s_2}(s_2-s)^{\alpha}dsd\tH^{n-1}(y),
    \end{equation*}
    where $s_1=\min\{s: y+su\in K\}$ and $s_2=\max\{s: y+su\in K\}$. Since $X_K(y,u)=s_2-s_1$, we have
    \begin{equation}\label{rdm-Xray}
        \rho_{\Ra K}(u)^{\alpha}=\frac{1}{(\alpha+1)|K|}\int_{K|u^{\perp}}X_K(y,u)^{\alpha+1}d\tH^{n-1}(y).
    \end{equation}
    Therefore, formula \eqref{dmv-RaK} follows from \eqref{chord-Xray}.
\end{proof}

We conclude this section with a basic computation concerning radial mean bodies, which also appear in \cite{GZ, HL3}. Here $K\triangle L$ denotes the symmetric difference of $K$ and $L$.

\begin{lem}\label{rdm-covariogram}
    Let $K\subset\rn$ be a convex body. Then
    \begin{align}
        \rho_{\Ra K}(u)^{\alpha}&=\frac{|\alpha|}{2|K|}\int_0^{\infty}r^{\alpha-1}|K\triangle (K+ru)|dr,&&\alpha\in(-1,0),\label{rdm-10}\\
        \rho_{\Ra K}(u)^{\alpha}&=\frac{\alpha}{|K|}\int_0^{\infty}r^{\alpha-1}|K\cap (K+ru)|dr,&&\alpha>0. \label{rdm0}
    \end{align}
\end{lem}

\begin{proof}
Formula \eqref{rdm0} is stated in \cite[Lemma 3.1]{GZ}. Here, it suffices to consider the case $\alpha\in(-1,0)$. Since
    \begin{equation*}
    \begin{aligned}
        \int_{K|u^{\perp}}X_K(x,u)^{\alpha+1}&d\tH^{n-1}(x)\\
        &=|\alpha|(\alpha+1)\int_{K|u^{\perp}}\Big(\int_0^{X_K(x,u)}r^{\alpha}dr+\int_{X_K(x,u)}^{\infty}r^{\alpha-1}X_K(x,u)dr\Big)d\tH^{n-1}(x)\\
        &=|\alpha|(\alpha+1)\int_{K|u^{\perp}}\int_0^{\infty}r^{\alpha-1}\min\{X_K(x,u), r\}drd\tH^{n-1}(x)\\
        &=|\alpha|(\alpha+1)\int_0^{\infty}r^{\alpha-1}\int_{K|u^{\perp}}\min\{X_K(x,u),r\}d\tH^{n-1}(x)dr\\
        &=\frac{|\alpha|(\alpha+1)}{2}\int_0^{\infty}r^{\alpha-1}|K\triangle (K+ru)|dr,
    \end{aligned}
\end{equation*}
the equality \eqref{rdm-10} follows from \eqref{rdm-Xray}.
\end{proof}

\begin{comment}
    For $\alpha>0$, note that
\begin{equation*}
    \begin{aligned}
        \int_{K|u^{\perp}}X_K(x,u)^{\alpha+1}dx
        &=\alpha(\alpha+1)\int_{K|u^{\perp}}\int_0^{X_K(x,u)}r^{\alpha-1}\Big(X_K(x,u)-r\Big)drd\tH^{n-1}(x)\\
        &=\alpha(\alpha+1)\int_0^{\infty}r^{\alpha-1}\int_{K|u^{\perp}} \Big(X_K(x,u)-r\Big)_+d\tH^{n-1}(x)dr,
        %&=\alpha(\alpha+1)\int_0^{\infty}r^{\alpha-1}|K\cap (K+ru)|dr,
    \end{aligned}
\end{equation*}
where $h(x)_+=\max\{h(x), 0\}$ denotes the positive part of the function $h(x)$. Hence
\[\int_{K|u^{\perp}}X_K(x,u)^{\alpha+1}d\tH^{n-1}(x)=\alpha(\alpha+1)\int_0^{\infty}r^{\alpha-1}|K\cap (K+ru)|dr,\]
together with \eqref{rdm-Xray}, we obtain \eqref{rdm0}.
\end{comment}

\section{Chord power integrals for functions}\label{chord f}

In this section, we extend the notion of chord power integrals from convex bodies to general non-negative measurable functions via level-set representations, which link the functional setting to the geometric counterpart.

Visintin \cite{Vis} pointed out that, as  a consequence of Fubini's theorem, a
 generalized coarea formula can be established. More precisely, for $\alpha\in(-1,0)$ and non-negative $f\in W^{-\alpha,1}(\rn)$, 
 \begin{equation}\label{Vis}
     \di\frac{|f(x)-f(y)|}{|x-y|^{n-\alpha}}dxdy=\frac{2n\omega_n}{|\alpha|(\alpha+1)}\int_0^{\infty}I_{\alpha+1}(\{f\ge t\})dt.
 \end{equation}

For $\alpha>0$ we have the following result.

\begin{thm}\label{lay}
    For $\alpha>0$ and log-concave $f\in L^1(\rn)$,
    \[\frac{\alpha(\alpha+1)}{n\omega_n}\di\frac{\min\{f(x),f(y)\}}{|x-y|^{n-\alpha}}dxdy=\int_0^{\infty}I_{\alpha+1}(\{f\ge t\})dt.\]
\end{thm}

\begin{proof}

Recall that for a convex body $K\subset\rn$ and $\alpha>0$,
\[I_{\alpha+1}(K)=\frac{\alpha(\alpha+1)}{n\omega_n}\di\frac{\chi_K(x)\chi_K(y)}{|x-y|^{n-\alpha}}dxdy.\]
By Fubini's theorem, we have
\[\int_0^{\infty}I_{\alpha+1}(\{f\ge t\})dt=\frac{\alpha(\alpha+1)}{n\omega_n}\di \int_0^{\|f\|_{\infty}}\frac{\chi_{\{f\ge t\}}(x)\chi_{\{f\ge t\}}(y)}{|x-y|^{n-\alpha}}dtdxdy.\]

Since
\[\min\{f(x), f(y)\}=\int_0^{\|f\|_{\infty}}\chi_{\{f\ge t\}}(x)\chi_{\{f\ge t\}}(y)dt,\]
we conclude the proof by substituting this identity.
\end{proof}

It remains to consider the endpoint cases $\alpha=-1$ and $\alpha=0$.
For non-negative $f\in W^{1,1}(\mathbb{R}^n)$, where
\[W^{1,1}(\rn)=\Big\{f\in L^1(\rn): |\nabla f|\in L^1(\rn)\Big\},\]
it follows from \cite[Theorem 3]{BBM} (see also \cite[Theorem 3]{HL1}) that
\begin{equation*}
\begin{aligned}
    \lim_{\alpha\rightarrow -1^+}(1+\alpha)\di \frac{|f(x)-f(y)|}{|x-y|^{n-\alpha}}dxdy&=2\omega_{n-1}\int_{\rn}|\nabla f(x)|dx\\
    &=2\omega_{n-1}\int_0^{\infty}\tH^{n-1}(\partial \{f\ge t\})dt,
\end{aligned}
\end{equation*}
where the last equality follows from the coarea formula \cite[Theorem 5.9]{EG}. Hence
\begin{equation}\label{I0}
    \lim_{\alpha\rightarrow -1^+}(1+\alpha)\di \frac{|f(x)-f(y)|}{|x-y|^{n-\alpha}}dxdy=2n\omega_n\int_0^{\infty}I_0(\{f\ge t\})dt.
\end{equation}

The following two lemmas show that the case $\alpha=0$ is also well-defined.

\begin{lem}\label{conv-Lp}
    For non-negative $f\in L^{\frac12}(\rn)\cap L^1(\rn)$,
    \[\lim_{\alpha\rightarrow 0^+}\|f\|_\frac{n}{n+\alpha}=\|f\|_1.\]
\end{lem}

\begin{proof}
    Let $A=\{f\ge 1\}$ and $B=\rn\backslash\{f\ge 1\}$. Assume that $(\alpha_k)$ is a positive sequence such that $\lim_{k\rightarrow \infty}\alpha_k=0$. 
    
    For $x\in A$, we have $f(x)^\frac{n}{n+\alpha_k}\leq f(x)$.
The dominated convergence theorem implies that
\begin{align*}
\lim_{k\to \infty} \int_A f(x)^\frac{n}{n+\alpha_k}dx= \int_A f(x) dx.   
\end{align*}

For $x\in B$, we have $f(x)^\frac{n}{n+\alpha_k}\leq f(x)^{1/2}$. By the dominated convergence theorem, we have
\begin{align*}
    \lim_{k\to \infty} \int_B f(x)^\frac{n}{n+\alpha_k}dx=\int_B f(x)dx,
\end{align*}
which completes the proof.
\end{proof}

\begin{lem}\label{I1}
Let  $f\in L^{\frac12}(\rn)\cap L^1(\rn)$ be non-negative. Then
\begin{align}
\lim_{\alpha \to 0^+} \alpha \int_{\rn}\int_{\rn} \frac{\min\{f(x),f(y)\}}{|x-y|^{n-\alpha}}dxdy=n\omega_n\|f\|_1.    
\end{align}
\end{lem}

\begin{proof}
By Lemma \ref{conv-Lp}, $f\in L^{\frac{n}{n+\alpha}}(\rn)$ for $\alpha$ sufficiently small. Hence Theorem A implies that
\[\limsup_{\alpha\rightarrow 0^+}\alpha\di \frac{\min\{f(x),f(y)\}}{|x-y|^{n-\alpha}}dxdy\leq \lim_{\alpha\rightarrow 0^+}\alpha\sigma_{n,\alpha}\|f\|_{\frac{n}{n+\alpha}}=n\omega_n\|f\|_1.\]
On the other hand, by the layer cake formula for $\min\{f(x), f(y)\}$ and Fubini's theorem,
\[\di\frac{\min \{f(x), f(y)\} }{|x-y|^{n-\alpha}}dxdy=\int_0^{\infty}\di \frac{ \chi_{ \{f\ge t\} }(x) \chi_{ \{f\ge t\} }(y) }{|x-y|^{n-\alpha}}dxdydt.\]

By \cite[Lemma 4.1]{Cai},
\[\lim_{\alpha\rightarrow 0}\alpha\di\frac{\chi_{ \{f\ge t\} }(x)\chi_{ \{f\ge t\} }(y)}{|x-y|^{n-\alpha}}dxdy=n\omega_n|\{f\ge t\}|.\]
Therefore, Fatou's Lemma implies that
\[n\omega_n\|f\|_1=\int_0^{\infty}n\omega_n|\{f\ge t\}|dt\leq \liminf_{\alpha\rightarrow 0^+}\alpha\di\frac{\min \{f(x), f(y)\} }{|x-y|^{n-\alpha}}dxdy,\]
which completes the proof.
\end{proof}

Recall that $I_1(K)=|K|$. Lemma \ref{I1} indicates that
\begin{equation}\label{I11}
    \lim_{\alpha\rightarrow 0^+}\frac{\alpha(\alpha+1)}{n\omega_n}\di \frac{\min\{f(x), f(y)\} }{|x-y|^{n-\alpha}}dxdy=\int_0^{\infty}I_1(\{f\ge t\})dt,
\end{equation}
for log-concave functions $f\in L^1(\mathbb{R}^n)$.

Theorem \ref{lay}, together with \eqref{Vis}, \eqref{I0} and \eqref{I11}, indicates a natural functional analogue of chord power integrals. Before introducing the extension, we first establish that the associated level-set representation is well-defined for log-concave, integrable functions for all $\alpha\ge -1$.

\begin{lem}\label{finiteness}
    Let $\alpha\ge-1$ and $f\in L^1(\rn)$ be a log-concave function. Then $f\in L^{\frac12}(\rn)$ and
    \[\int_0^{\|f\|_{\infty}}I_{\alpha+1}(\{f\ge t\})dt<\infty.\]
\end{lem}

\begin{proof}
    Since $f$ is log-concave and integrable, there is a convex function $v:\rn\rightarrow (-\infty,\infty]$ such that $f(x)=e^{-v(x)}$ and $\lim_{|x|\rightarrow\infty}v(x)=\infty$. By \cite[Lemma 2.5]{CF}, there exist constants $a,b\in\tr$ with $a>0$ such that
    \begin{equation}\label{L1log}
        v(x)\ge a|x|+b,
    \end{equation}
    for all $x\in\rn$. Then
    \[f(x)^{1/2}\leq e^{-\frac{1}{2}(a|x|+b)},\]
    where the right-hand side is integrable. Hence $f\in L^{\frac12}(\rn)$.

    Moreover, \eqref{L1log} also implies that
    \begin{equation}\label{inclusion}
        \{f\ge t\}\subset B_t:=\{x\in \rn: a|x|+b\leq -\log t\}.
    \end{equation}
    We first consider $\alpha>-1$. Since $I_{\alpha+1}(\lambda K)=\lambda^{n+\alpha}I_{\alpha+1}(K)$, the inclusion \eqref{inclusion} implies that
    \begin{equation*}
        \begin{aligned}
            \int_0^{\|f\|_{\infty}}I_{\alpha+1}(\{f\ge t\})dt\leq \int_0^{\|f\|_\infty}I_{\alpha+1}(B_t)dt&=\frac{I_{\alpha+1}(B^n)}{a^{n+\alpha}}\int_0^{\|f\|_{\infty}}(-\log t-b)^{n+\alpha}dt\\
            &=\frac{I_{\alpha+1}(B^n)}{a^{n+\alpha}e^b}\int_{-\log\|f\|_{\infty}-b}^{\infty}s^{n+\alpha}e^{-s}ds<\infty.
        \end{aligned}
    \end{equation*}

For $\alpha=-1$, by the Cauchy's integral formula in \eqref{CCPH},
    \begin{equation*}
        \int_0^{\|f\|_{\infty}}I_0(\{f\ge t\})dt=\frac{\omega_{n-1}}{n\omega_n}\int_0^{\|f\|_{\infty}}\tH^{n-1}(\partial \{f\ge t\})dt.
    \end{equation*}
    Since $\{f\ge t\}$ is convex, and $\{f\ge t\}\subset B_t$ as shown in \eqref{inclusion}, we have
    \[\tH^{n-1}(\partial \{f\ge t\})\leq \tH^{n-1}(\partial B_t)=\frac{n\omega_n}{a^{n-1}}(-\log t-b)^{n-1}.\]
    This implies that
    \begin{equation*}
       \begin{aligned}
            \int_0^{\|f\|_\infty}I_0(\{f\ge t\})dt&\leq \frac{\omega_{n-1}}{a^{n-1}}\int_0^{\|f\|_{\infty}}(-\log t-b)^{n-1}dt\\&=\frac{\omega_{n-1}}{a^{n-1}e^b}\int_{-\log\|f\|_{\infty}-b}^{\infty}s^{n-1}e^{-s}ds<\infty,
        \end{aligned}
    \end{equation*}
    which concludes the proof.
\end{proof}

With this finiteness result in hand, we now introduce the definition of the chord power integral for log-concave, integrable functions.

\begin{defn}\label{def-I_qf}
    Let $\alpha\ge -1$ and $f\in L^1(\rn)$ be a log-concave function. The chord power integral for $f$ is defined by
    \[I_{\alpha+1}(f)=\int_0^{\|f\|_{\infty}}I_{\alpha+1}(\{f\ge t\})dt.\]
\end{defn}

\begin{rem}\label{rem-ex}
    The definition can be extended to a broader class when $\alpha\in(-1,n)$. For $\alpha\in(-1,0)$,
    \[I_{\alpha+1}(f)=\frac{|\alpha|(\alpha+1)}{2n\omega_n}\di\frac{|f(x)-f(y)|}{|x-y|^{n-\alpha}}dxdy,\]
    is finite for $f\in W^{-\alpha,1}(\rn)$. For $\alpha\in(0,n)$, Theorem A implies that
    \[I_{\alpha+1}(f)=\frac{\alpha(\alpha+1)}{n\omega_n}\di \frac{ \min\{f(x), f(y)\} }{|x-y|^{n-\alpha}}dxdy\]
    is finite for non-negative $f\in  L^{\frac{n}{n+\alpha}}(\rn)$. This observation suggests possible extensions of chord power integrals from convex bodies to general Borel sets.
\end{rem}

By Definition \ref{def-I_qf}, Lemma \ref{I1}, \eqref{I0} and \eqref{CCPH}, we have the extended Cauchy's integral formula, Crofton's volume formula,  and Poincar\'e-Hadwiger's integral formula as follows. Here $f\in L^1(\rn)$ is log-concave.
\vskip7pt

\noindent{\bf Cauchy's Integral Formula:}
\begin{equation*}
    \begin{aligned}
        I_0(f)=\frac{\omega_{n-1}}{n\omega_n}\int_0^{\|f\|_{\infty}}|\nabla f(x)|dx,
    \end{aligned}
\end{equation*}
provided $f\in W^{1,1}(\rn)$.
\vskip 7pt

\noindent{\bf Crofton's Integral Formula:}
\[I_1(f)=\|f\|_1.\]

\vskip 7pt

\noindent{\bf Poincar\'e-Hadwiger's Integral Formula:}
\begin{equation*}
    \begin{aligned}
        I_{n+1}(f)=\int_0^{\|f\|_{\infty}}I_{n+1}(\{f\ge t\})dt&=\frac{n+1}{\omega_n}\int_0^{\|f\|_{\infty}}|\{f\ge t\}|^2dt\\
        &=\frac{n+1}{\omega_n}\di \min\{f(x), f(y)\}dxdy.
    \end{aligned}
\end{equation*}

Moreover, by Definition \ref{def-I_qf}, the fractional Sobolev inequality \eqref{frac-Sob} and the chord Sobolev inequality \eqref{chord-Sob} admit the following unified representation in terms of $I_{\alpha+1}(f)$.
\vskip 7pt

\noindent{\bf Functional chord isoperimetric inequalities.} For a log-concave, integrable function $f$,
\begin{equation*}
    \begin{aligned}
        I_{\alpha+1}(f)&\ge \tilde{\sigma}_{n,\alpha}\|f\|_{\frac{n}{n+\alpha}},\quad \alpha\in(-1,0),\\
        I_{\alpha+1}(f)&\le \tilde{\sigma}_{n,\alpha}\|f\|_{\frac{n}{n+\alpha}},\quad \alpha\in(0,n),
    \end{aligned}
\end{equation*}
and
\[I_{\alpha+1}(f)\ge \tilde{\sigma}_{n,\alpha}\|f\|_1^{\frac{n+\alpha}{n}}\|f\|_{\infty}^{-\frac{\alpha}{n}},\quad\alpha>n,\]
where the $\tilde{\sigma}_{n,\alpha}$ is given by \eqref{tilde sigma}.

\section{Radial mean bodies for functions}\label{rdm-fcn}

In this section, we study radial mean bodies in the functional setting and develop a formulation adapted to our framework. We focus in particular on the case $\alpha=0$, introducing a functional object $\rR_0 f$ that arises as a limiting quantity and is adapted to the analysis of Theorem C in Sections \ref{entropy} and \ref{log Sob ineq}. Our approach emphasizes a direct and transparent representation compatible with the level-set structure developed earlier.

As shown in Lemma \ref{chord-RaK},
\[I_{\alpha+1}(K)=\frac{(\alpha+1)|K|}{n\omega_n}\int_{\sn}\rho_{\Ra K}(u)^{\alpha}du,\]
where $\Ra K$ is the radial $\alpha$-th mean body of a convex body $K\subset \rn$. For a non-negative function $f$, it follows from Definition \ref{def-I_qf} that
\begin{equation*}
\begin{aligned}
    I_{\alpha+1}(f)&=\int_0^{\|f\|_{\infty}}I_{\alpha+1}(\{f\ge t\})dt=\frac{\alpha+1}{n\omega_n}\int_0^{\|f\|_{\infty}}|\{f\ge t\}|\int_{\sn}\rho_{\Ra \{f\ge t\}}(u)^{\alpha}dudt.
\end{aligned}
\end{equation*}
Assume that $f$ is non-zero. We define a probability measure $\mu_f$ on $(0,\|f\|_\infty)$ by
\[d\mu_f(t)=\frac{|\{f\ge t\}|}{\|f\|_1}dt.\]
Then
\begin{equation}\label{If-Raf}
    I_{\alpha+1}(f)=\frac{(\alpha+1)\|f\|_1}{n\omega_n}\int_{\sn}\int_0^{\|f\|_{\infty}}\rho_{\Ra \{f\ge t\}}(u)^{\alpha}d\mu_f(t)du.
\end{equation}

This motivates the functional extensions of radial mean bodies defined below. We adopt the notation $\Raf$ in accordance with Haddad and Ludwig~\cite{HL3}, where related objects were introduced in the context of fractional $L^2$ polar projection bodies and affine Hardy–Littlewood–Sobolev inequalities.

\begin{defn}\label{rdm-f}
    Let $f\in L^1(\rn)$ be a non-zero, log-concave function. For $\alpha>-1$ and $\alpha\neq 0$, 
    \[\rho_{\Raf}(u)^{\alpha}=\int_0^{\|f\|_{\infty}}\rho_{\Ra\{f\ge t\}}(u)^{\alpha}d\mu_f(t),\]
    and for $\alpha =0$,
    \[\log\rho_{\rR_0f}(u)=\int_0^{\|f\|_{\infty}}\log\rho_{\rR_0\{f\ge t\}}(u)d\mu_f(t).\]
\end{defn}

We note that Langharst, Mar\'in Sola and Ulivelli~\cite{LSU} first developed a higher-order functional theory of radial mean bodies in the setting of upper semicontinuous, log-concave functions, which provides a general framework for such objects in that context. The present definition coincides with their construction, as shown in Proposition~\ref{ex-rdm}.

In this section, we aim to extend the scope of the theory to a broader class of functions in $L^1(\rn)$, with particular emphasis on the body $\rR_0 f$, which plays a central role in Theorem~C and its geometric interpretation. Unlike the higher-order setting of~\cite{LSU}, the level-set representation adopted here allows a more direct treatment of this quantity.

\begin{prop}\label{ex-rdm}
    Let $f\in L^1(\rn)$ be non-zero, log-concave. Then 
    \begin{equation}\label{fppb}
        \rho_{\Raf}(u)^{\alpha}=\frac{|\alpha|}{2\|f\|_1}\int_0^{\infty}r^{\alpha-1}\int_{\rn}|f(x)-f(x+ru)|dxdr
    \end{equation}
    for $\alpha\in(-1,0)$ and
    \begin{equation}\label{laf}
        \rho_{\Raf}(u)^{\alpha}=\frac{\alpha}{\|f\|_1}\int_0^{\infty}r^{\alpha-1}\int_{\rn}\min \{f(x), f(x+ru)\}dxdr
    \end{equation}
    for $\alpha>0$.
\end{prop}

\begin{proof}
   %For simplicity of notation, we denote by $F_t:=\{f\ge t\}$. 
   First, we consider $\alpha\in (-1,0)$. By Definition \ref{rdm-f}, we have
   \begin{equation*}
       \rho_{\Raf}(u)^{\alpha}=\int_0^{\|f\|_\infty}\rho_{\Ra \{f\ge t\} }(u)^{\alpha}d\mu_f(t),
   \end{equation*}
   where $d\mu_f(t)=\frac{|\{f\ge t\}|}{\|f\|_1}dt$.

   It follows from Lemma \ref{rdm-covariogram} that
    \begin{equation*}
        \begin{aligned}
            %\rho_{\Raf}(u)^{\alpha}&=
            \int_0^{\|f\|_\infty}\rho_{\Ra \{f\ge t\} }(u)^{\alpha}d\mu_f(t)
            &=\frac{|\alpha|}{2\|f\|_1}\int_0^{{\|f\|_\infty}}
            \int_0^{\infty}r^{\alpha-1}| \{f\ge t\} \triangle(\{f\ge t\}+ru)|drdt\\
            &=\frac{|\alpha|}{\|f\|_1}\int_0^{\infty}r^{\alpha-1}\int_0^{\|f\|_\infty}\int_{\rn}|\chi_{ \{f\ge t\} }(x)-\chi_{ \{f\ge t\}}(x+ru)|dxdtdr\\
            &=\frac{|\alpha|}{2\|f\|_1}\int_0^{\infty}r^{\alpha-1}\int_{\rn}|f(x)-f(x+ru)|dxdr,
        \end{aligned}
    \end{equation*}
    which proves \eqref{fppb}.

    For $\alpha>0$, by Definition \ref{rdm-f} and Lemma \ref{rdm-covariogram}, we obtain
    \begin{equation*}
        \begin{aligned}
            \rho_{\Raf}(u)^{\alpha}&=\int_0^{\|f\|_{\infty}}\rho_{\Ra \{f\ge t\} }(u)^{\alpha}d\mu_f(t)\\
            &=\frac{\alpha}{\|f\|_1}\int_0^{\|f\|_{\infty}}\int_0^{\infty}r^{\alpha-1}|\{f\ge t\}\cap (\{f\ge t\}+ru)|drdt\\
            &=\frac{\alpha}{\|f\|_1}\int_0^{\infty}r^{\alpha-1}\int_{\rn}\min\{f(x), f(x+ru)\}dxdr,
        \end{aligned}
    \end{equation*}
    which concludes the proof.
\end{proof}

For $\alpha\in(-1,0)$, Remark~\ref{rem-ex} shows that $I_{\alpha+1}(f)$ is finite for non-negative, non-zero $f\in W^{-\alpha ,1}(\rn)$. By \eqref{If-Raf}, it follows that $\rho_{\Raf}(u)$ is finite for almost all $u$. Moreover, Haddad and Ludwig~\cite[Proposition 2]{HL1} proved that $\Ra f$ is a star body in this case.

For $\alpha\in (0,n)$, the chord Sobolev inequality \eqref{chord-Sob} implies that $I_{\alpha+1}(f)$ is finite for non-negative, non-zero $f\in L^{\frac{n}{n+\alpha}}(\rn)$. Hence $\rho_{\Raf}(u)$ is finite for almost all $u$.

We now turn to the case $\alpha=0$. Unlike the case $\alpha\neq0$, the definition does not directly yield an explicit representation of $\rR_0f$. We begin with the following proposition, which provides an explicit expression for $\rho_{\rR_0 K}$. This formula can also be obtained from \cite[Lemma 16]{HL3} under slightly different assumptions; we include the proof for completeness.

\begin{prop}\label{R0K-explicit}
    Let $K\subset\rn$ be a convex body. For $u\in\sn$,
    \[\log\rho_{\rR_0 K}(u)=-\gamma+\int_0^{\infty}\frac{1}{r}\Big(\frac{|K\cap (K+ru)|}{|K|}-e^{-r}\Big)dr,\]
    where $\gamma=-\Gamma^{\prime}(1)$ is the Euler constant.
\end{prop}

\begin{proof}
    For $u\in\sn$, we let $g_K(r)=|K\cap (K+ru)|/|K|$ and $\omega(r)=e^{-r}$. By Proposition \ref{ex-rdm}, we have
    \begin{equation}
        \frac{\rho_{\Ra K}(u)}{\Gamma(\alpha+1)^{1/\alpha}}=\Big(\frac{\int_0^{\infty}r^{\alpha-1}g_K(r)dr}{\int_0^{\infty}r^{\alpha-1}\omega(r)dr}\Big)^{1/\alpha}
    \end{equation}
    for $\alpha>0$. Since $\rho_{\Ra K}(u)$ and $\Gamma(\alpha+1)^{1/\alpha}$ are continuous functions with respect to $\alpha$ over $(-1,\infty)$, we have
    \[e^{\gamma}\rho_{\rR_0 K}(u)=\lim_{\alpha\rightarrow 0^+}\Big(\frac{\int_0^{\infty}r^{\alpha-1}g_K(r)dr}{\int_0^{\infty}r^{\alpha-1}\omega(r)dr}\Big)^{1/\alpha}\]

    Moreover, we have
    \begin{equation}\label{omega}
        \int_0^{\infty} r^{\alpha-1}\omega(r)dr=\int_{r_0}^{\infty}r^{\alpha-1}\omega(r)dr+\int_{0}^{r_0}r^{\alpha-1}(1-\omega(r))dr+\frac{r_0^{\alpha}}{\alpha}
    \end{equation}
    and
    \begin{equation}\label{g}
        \int_0^{\infty} r^{\alpha-1}g_K(r)dr=\int_{r_0}^{\infty}r^{\alpha-1}g_K(r)dr+\int_{0}^{r_0}r^{\alpha-1}(1-g_K(r))dr+\frac{r_0^{\alpha}}{\alpha}.
    \end{equation}

    By \eqref{omega} and \eqref{g}, and the elementary relation
    \[\lim_{\alpha\rightarrow 0^+}\bigg(\frac{a(\alpha)+\frac{c^{\alpha}}{\alpha}}{b(\alpha)+\frac{c^{\alpha}}{\alpha}}\bigg)^{1/\alpha}=\exp{(a(0)-b(0))},\]
    which is valid for a constant $c\neq 0$ and continuous functions $\alpha\mapsto a(\alpha)$ and $\alpha\mapsto b(\alpha)$, we have
    \[e^{\gamma}\rho_{\rR_0 K}(u)=\exp\Big(\int_0^{\infty}\frac{g_K(r)-\omega(r)}{r}dr\Big),\]
    which concludes the proof.
\end{proof}

Using the above result, we can now derive an explicit representation for $\rR_0 f$.

\begin{lem}
    Let $f\in L^1(\rn)$ be non-zero, log-concave. For $u\in\sn$,
    \begin{equation}\label{log-R0}
        \log \rho_{\rR_0f}(u)=-\gamma+\int_0^{\infty}\frac{1}{r}\Big(\frac{1}{\|f\|_1}\int_{\rn}\min\{f(x), f(x+ru)\}dx-e^{-r}\Big)dr.
    \end{equation}
\end{lem}

\begin{proof}
    By Proposition \ref{R0K-explicit} and Definition \ref{rdm-f},
    \begin{equation*}
        \begin{aligned}
            \log\rho_{\rR_0 f}(u)&=\int_0^{\|f\|_{\infty}}\log \rho_{\rR_0 \{f\ge t\}}(u)d\mu_f(t)\\
            &=\int_0^{\|f\|_{\infty}}\frac{|\{f\ge t\}|}{\|f\|_1}\Big(-\gamma +\int_0^{\infty}\frac{1}{r}\Big(\frac{|\{f\ge t\}\cap (\{f\ge t\}+ru)|}{|\{f\ge t\}|}-e^{-r}\Big)dr\Big)dt.
        \end{aligned}
    \end{equation*}
    It follows from Fubini's theorem that
    \begin{equation*}
    \begin{aligned}
        \log{\rho_{\rR_0 f}}(u)&=-\gamma+\int_0^{\infty}\frac{1}{r}\Big(\frac{1}{\|f\|_1}\int_0^{\|f\|_{\infty}}|\{f\ge t\}\cap (\{f\ge t\}+ru)|dt-e^{-r}\Big)dr\\
        &=-\gamma+\int_0^{\infty}\frac{1}{r}\Big(\frac{1}{\|f\|_1}\int_{\rn}\min\{f(x), f(x+ru)\}dx-e^{-r}\Big)dr,
    \end{aligned}
    \end{equation*}
    which completes the proof.
\end{proof}

We conclude this section by establishing that the expression in \eqref{log-R0} extends to a broader class of functions. While it is initially defined for log-concave, integrable functions, we show that it remains well-defined for non-negative, non-zero functions $f\in W^{\beta,1}(\rn)$, where $\beta\in(0,1)$ is fixed.

Before proceeding, we establish a preliminary lemma, whose notation and conclusion will also be used in Section \ref{log Sob ineq}.

\begin{lem}\label{E-def}
    Let $\beta\in(0,1)$ and $f\in W^{\beta,1}(\rn)$ be a non-negative, non-zero function. Then
    \begin{equation*}
        \mathcal{E}_u(f):=\int_0^1\frac{1}{r}\int_{\rn}|f(x)-f(x+ru)|dxdr
    \end{equation*}
    is finite for all $u\in\sn$ and integrable over $\sn$.
\end{lem}

\begin{proof}
    Since $\beta\in(0,1)$, we have
    \begin{equation}\label{E-finite}
    \begin{aligned}
        \mathcal{E}_u(f)&\le \int_0^1 r^{-\beta-1}\int_{\rn}|f(x)-f(x+ru)|dxdr\\
        &\le \int_0^{\infty}r^{-\beta-1}\int_{\rn}|f(x)-f(x+ru)|dxdr=\frac{2\|f\|_1}{\beta}\rho_{\rR_{-\beta}f}(u)^{-\beta},
    \end{aligned}
    \end{equation}
    where the last equality is due to Proposition \ref{ex-rdm}. Haddad and Ludwig \cite[Proposition]{HL1} proved that $\rR_{-\beta}f$ is a star body, which shows that $\mathcal{E}_u(f)$ is finite for all $u\in\sn$. 

    Moreover, since $f\in W^{\beta,1}(\rn)$, inequality \eqref{E-finite} implies that
    \begin{equation}\label{E-integrable}
        \begin{aligned}
            \int_{\sn}\mathcal{E}_u(f)du&\leq \int_{\sn}\int_{0}^{\infty}r^{-\beta-1}\int_{\rn}|f(x)-f(x+ru)|dxdrdu\\
            &=\di \frac{|f(x)-f(y)|}{|x-y|^{n+\beta}}dxdy<\infty,
        \end{aligned}
    \end{equation}
    which concludes the proof.
\end{proof}

\begin{lem}\label{Q-def}
    Let $f\in L^{\frac{1}{2}}(\rn) \cap L^1(\rn)$ be non-negative and non-zero. Then
    \begin{equation*}
        \mathcal{Q}_u(f):=\int_1^{\infty}\frac{1}{r}\int_{\rn}\min\{f(x), f(x+ru)\}dxdr
    \end{equation*}
    is finite for almost all $u\in\sn$ and integrable over $\sn$.
\end{lem}

\begin{proof}
    By Lemma \ref{conv-Lp}, $f\in L^{\frac{\alpha^{\prime}}{n+\alpha^{\prime}}}(\rn)$ for some $\alpha^{\prime}\in (0,1)$. Then the chord Sobolev inequality \eqref{chord-Sob} implies that
    \begin{equation*}
        \begin{aligned}
            \int_{\sn} \mathcal{Q}_u(f)du&= \int_{\sn}\int_1^{\infty}\frac{1}{r}\int_{\rn}\min\{f(x), f(x+ru)\}dxdrdu\\
            &\leq \int_{\sn}\int_0^{\infty}r^{\alpha^{\prime}-1}\int_{\rn}\min\{f(x), f(x+ru)\}dxdrdu\\
            &=\di \frac{ \min\{f(x), f(y)\} }{|x-y|^{n-\alpha^{\prime}}}dxdy\leq \sigma_{n,\alpha^{\prime}}\|f\|_{\frac{n}{n+\alpha^{\prime}}}<\infty.
        \end{aligned}
    \end{equation*}
    This shows that $\mathcal{Q}_u(f)$ is integrable over $\sn$ and hence finite almost everywhere.
\end{proof}

\begin{lem}\label{ex-R0f}
    Let $\beta\in(0,1)$ be fixed, and $f\in W^{\beta,1}(\rn)$ be a non-negative, non-zero  function.
    Then
    \begin{equation}\label{log0-statement}
        \log\rho_{\rR_0f}(u)=-\gamma+\int_0^{\infty}\frac{1}{r}\Big(\frac{1}{\|f\|_1}\int_{\rn}\min\{f(x), f(x+ru)\}dx-e^{-r}\Big)dr
    \end{equation}
    is well-defined for all $u\in\sn$.

    If, in addition, $f\in L^{\frac12}(\rn)$, then $\log\rho_{\rR_0f}(u)$ is finite for almost all $u\in\sn$. Moreover, $\log\rho_{\rR_0f}\in L^1(\sn)$.
\end{lem}

\begin{proof}
    Let $u\in\sn$ and $g(r,u)=\frac{1}{\|f\|_1}\int_{\rn}\min\{f(x), f(x+ru)\}dx$. By \eqref{log0-statement}, we have
    \begin{equation}\label{log0-expansion}
        \log\rho_{\rR_0f}(u)=C+\int_0^1\frac{g(r,u)-1}{r}dr+\int_1^{\infty}\frac{g(r,u)}{r}dr,
    \end{equation}
    where $C$ is a constant given by
    \[C:=-\gamma+\int_0^1\frac{1-e^{-r}}{r}dr-\int_1^{\infty}\frac{e^{-r}}{r}dr.\]

    We first consider the integral over $(0,1)$ in \eqref{log0-expansion}. Note that
%    \begin{equation*}
%        \begin{aligned}
%            \int_0^1\frac{g(r,u)-1}{r}dr&=-\frac{1}{\|f\|_1}\int_0^1\frac{1}{r}\int_{\rn}(\|f\|_1-\min\{f(x),f(x+ru)\})dxdr\\
%            &=-\frac{1}{2\|f\|_1}\int_0^1\frac{1}{r}\int_{\rn}\Big((f(x+ru)-f(x))_++(f(x)-f(x+ru))_+\Big)dxdr,
%        \end{aligned}
%    \end{equation*}
%    where $f(x)_+:=\max \{0, f(x)\}$ denotes the positive part of $f$. Thus we have
    \begin{equation}\label{gE}
        \int_0^1\frac{g(r,u)-1}{r}dr=-\frac{1}{2\|f\|_1}\int_0^1 \int_{\rn}|f(x)-f(x+ru)|dxdr=-\frac{1}{2\|f\|_1}\mathcal{E}_u(f),
    \end{equation}
    where $\mathcal{E}_u(f)$ is defined in Lemma \ref{E-def} and is finite for all $u\in\sn$. Therefore $\log\rho_{\rR_0f}(u)$ is well-defined as an extended number.

    Moreover, by \eqref{log0-expansion}, \eqref{gE} and the quantity $\mathcal{Q}_u(f)$ defined in Lemma \ref{Q-def}, we have
    \begin{equation}\label{0EQ}
        \log\rho_{\rR_0f}(u)=C-\frac{1}{2\|f\|_1}\mathcal{E}_u(f)+\frac{1}{\|f\|_1}\mathcal{Q}_u(f).
    \end{equation}
    Therefore, by Lemma \ref{E-def} and Lemma \ref{Q-def}, $\log\rho_{\rR_0f}(u)$ is finite almost everywhere and integrable over $\sn$, provided that $f\in W^{\beta,1}(\rn)\cap L^{\frac12}(\rn)$ is non-negative, non-zero.
\end{proof}

\section{Applications to entropies of random lines}\label{entropy}

In the geometric chord isoperimetric inequalities \eqref{iso-chord-011} and \eqref{iso-chord-1n+11}, the inequality degenerates to an identity at the endpoints $\alpha=0$ and $\alpha=n$. Motivated by this phenomenon, Xu \cite{Xu} and Zhang \cite{Zh3} introduced entropy inequalities to capture nontrivial information in these regimes. 

In this section, we consider functional extensions of the chord entropy inequalities arising from the endpoint behavior of chord Sobolev inequalities. In particular, Theorem~C extends a functional entropy inequality to a broader class of functions, which will be further developed in the next section.

By Lemma~\ref{conv-Lp} and Lemma~\ref{I1}, the chord Sobolev inequality \eqref{chord-Sob} degenerates to an identity as $\alpha\to 0^+$ for non-negative $f\in L^{\frac12}(\rn)\cap L^1(\rn)$. Moreover, Maz${}'$ya and Shaposhnikova~\cite{MS} established the limiting behavior as $\alpha\to 0^-$ for fractional Sobolev inequalities: for $f\in \bigcup_{-1<\alpha<0}W^{-\alpha,1}(\rn)$,
\begin{equation}\label{conv-MS}
    \lim_{\alpha\rightarrow 0^-}-\alpha\int_{\rn}\int_{\rn}\frac{|f(x)-f(y)|}{|x-y|^{n-\alpha}}dxdy
    = n\omega_n\|f\|_1.
\end{equation}
Thus, the fractional Sobolev inequality \eqref{frac-Sob} also reduces to an identity in the limit $\alpha\to 0^-$. 

This suggests considering the endpoint case $\alpha=0$, which leads to functional extensions of the chord entropy inequalities established by Xu \cite{Xu} and Zhang \cite{Zh3}. For $\alpha=n$, the chord Sobolev inequality reduces to a trivial inequality; nevertheless, in order to retain the parallel with the geometric theory, we also formulate the corresponding entropy inequality.

We begin by recalling the definition of the chord entropy of a convex body $K\subset\rn$. We follow \cite{Zh3} with a slightly different normalization.
\vskip 2pt

\noindent{\bf The {$\boldsymbol{(\alpha+1)}$}-th chord entropy:} For a convex body $K\subset\rn$ and $\alpha>-1$, the $(\alpha+1)$-th chord entropy of $K$ is defined by
\[\Ent_{\alpha+1}(K)=-\frac{\alpha+1}{I_{\alpha+1}(K)}\int_{K\cap l\neq \emptyset} |K\cap l|^{\alpha+1}\log |K\cap l|dl.\]

In this paper, we will only use the endpoint cases $\alpha=0$ and $\alpha=n$, namely, 
\[\Ent_1(K)=-\frac{1}{|K|}\int_{K\cap l\neq \emptyset}|K\cap l|\log |K\cap l|dl,\]
and
\[\Ent_{n+1}(K)=-\frac{\omega_n}{|K|^2}\int_{K\cap l\neq\emptyset}|K\cap l|^{n+1}\log |K\cap l|dl,\]
where the coefficients are from the Crofton's volume formula and the Poincar\'e-Hadwiger's integral formula in \eqref{CCPH}, respectively.

The chord entropy inequalities for $\Ent_1(K)$ and $\Ent_{n+1}(K)$ state that
\begin{equation}\label{1}
    \Ent_1(K)\ge \Ent_1(B_K)=\Ent_1(B^n)+\frac{1}{n}\log\Big(\frac{\omega_n}{|K|}\Big),
\end{equation}
and
\begin{equation}\label{n+1}
    \Ent_{n+1}(K)\leq \Ent_{n+1}(B_K)=\Ent_{n+1}(B^n)+\frac{n+1}{n}\log\Big(\frac{\omega_n}{|K|}\Big),
\end{equation}
where $B_K$ denotes the ball with the same volume as $K$. Both inequalities are sharp and equality holds if and only if $K=B_K$ (see \cite{Xu, Zh3}).

\subsection{The chord entropy inequality for \texorpdfstring{$\Ent_1(f)$}{}}\hspace{\fill}

We first extend the geometric chord entropy $\Ent_1(K)$ to non-zero, integrable, log-concave functions by defining
\begin{equation}\label{E1-def}
    \Ent_1(f)=\int_0^{\infty}\Ent_1(\{f\ge t\})d\mu_f(t)-\frac{1}{n\|f\|_1}\int_{\rn}f(x)\log f(x)dx.
\end{equation}
The second term coincides with the classical entropy of $f$ in information theory, which arises from the  limiting behavior of $\|f\|_{\frac{n}{n+\alpha}}$, as shown by  \eqref{ent f}. 

By \cite[Proposition I.2]{BobkovMadiman}, the entropy $\int_{\rn}f(x)\log f(x)dx$ is finite, provided that $f\in L^1(\rn)$ is log-concave. We aim to show that $\Ent_1(f)$ is finite for  such $f$. 

We need the following lemma established by Xu \cite[Theorem 3.2]{Xu}. Since our normalization differs slightly from that in \cite{Xu}, we provide a proof for completeness.

\begin{lem}\cite[Theorem 3.2]{Xu}\label{Xu-3.2}
    For a convex body $K$ in $\rn$,
    \[\Ent_1(K)=-1-\frac{1}{n\omega_n}\int_{\sn}\log\rho_{\rR_0 K}(u)du.\]
\end{lem}

\begin{proof}
    By the definition of $\Ent_1(K)$, we have
    \begin{equation*}
        \begin{aligned}
            \Ent_1(K)&=-\frac{1}{|K|}\int_{\Aff}|K\cap l|\log |K\cap l|dl\\
            &=-\frac{1}{n\omega_n|K|}\int_{\sn}\int_{K|u^{\perp}}X_K(y,u)\log X_K(y,u)d\tH^{n-1}(y)du.
        \end{aligned}
    \end{equation*}
    Note that
    \begin{equation*}
        \begin{aligned}
            X_K(y,u)\log X_K(y,u)=\int_0^{X_K(y,u)}(\log s+1)ds=\int_{s_0}^{s_1}(\log (s-s_0)+1)ds,
        \end{aligned}
    \end{equation*}
    where $s_0=\min\{s: y+su\in K\}$ and $s_1=\max\{s: y+su\in K\}$. Hence we have
    \begin{equation*}
        \begin{aligned}
            \Ent_1(K)&=-\frac{1}{n\omega_n|K|}\int_{\sn}\int_{K|u^{\perp}}\int_{s_0}^{s_1}(\log (s-s_0) +1)dsd\tH^{n-1}(y)du\\
            &=-\frac{1}{n\omega_n|K|}\int_{\sn}\int_K\log\rho_{K-x}(u)dxdu-1=-\frac{1}{n\omega_n}\int_{\sn}\log \rho_{\rR_0 K}(u)du-1,
        \end{aligned}
    \end{equation*}
    where the last equality is by the definition of $\rR_0 K$.
\end{proof}

The following lemma extends Lemma \ref{Xu-3.2} from convex bodies to log-concave functions. Combined with Lemma \ref{finiteness} and Lemma \ref{ex-R0f}, the following result shows that $\Ent_1(f)$ is finite for non-zero, log-concave $f\in L^1(\rn)$.

\begin{lem}\label{ex-Xu3.2}
    For non-zero, log-concave $f\in L^1(\rn)$,
    \[\int_0^{\infty}\Ent_1(\{f\ge t\})d\mu_f(t)=-1-\frac{1}{n\omega_n}\int_{\sn}\log\rho_{\rR_0f}(u)du.\]
\end{lem}

\begin{proof}
    For log-concave $f\in L^1(\rn)$, Lemma \ref{Xu-3.2} implies that
    \begin{equation*}
        \begin{aligned}
            \int_0^{\infty}\Ent_1(\{f\ge t\})d\mu_f(t)&=\int_0^{\infty}\Big(-1-\frac{1}{n\omega_n}\int_{\sn}\log\rho_{\rR_0 \{f\ge t\} }(u)du\Big)d\mu_f(t)\\
            &=-1-\frac{1}{n\omega_n}\int_{\sn}\int_0^{\infty}\log\rho_{\rR_0 \{f\ge t\} }(u)d\mu_f(t)du.
        \end{aligned}
    \end{equation*}
    By Definition \ref{rdm-f}, we then have
    \[\int_{0}^{\infty}\Ent_1(\{f\ge t\})d\mu_f(t)=-1-\frac{1}{n\omega_n}\int_{\sn}\log\rho_{\rR_0f}(u)du,\]
    which concludes the proof.
\end{proof}

Before giving the chord entropy inequality for $\Ent_1(f)$, we first recall the following basic result, which is also known as the non-negativity of the relative entropy in information theory. For completeness, we provide the proof here.

\begin{lem}\label{rel-ent}
    For two probability density functions $h,l:\rn\rightarrow (0,\infty)$,
    \[\int_{\rn}h(x)\log\Big(\frac{h(x)}{l(x)}\Big)dx\ge 0.\]
    Equality holds if and only if $h=l$.
\end{lem}

\begin{proof}
    Since $h(x)dx$ is a probability measure on $\rn$, Jensen's inequality implies that
    \begin{equation*}
        \begin{aligned}
            \int_{\rn}h(x)\log\Big(\frac{h(x)}{l(x)}\Big)dx&=-\int_{\rn}h(x)\log\Big(\frac{l(x)}{h(x)}\Big)dx\\
            &\ge-\log\Big(\int_{\rn}\frac{l(x)}{h(x)} h(x)dx\Big)=-\log\|l\|_1=0.
        \end{aligned}
    \end{equation*}
    Equality holds if and only if $l(x)=ch(x)$ for almost all $x\in\rn$, where $c$ is a positive constant. Since $\|h\|_1=\|l\|_1=1$, we have $c=1$. Hence equality holds if and only if $h=l$.
\end{proof}

We also need the following result, which will be used to prove the chord entropy inequality.

\begin{prop}\label{E1cf}
   Let $f \in L^{1}(\rn)$ be non-zero, log-concave, and $f_{\lambda}(x)=f(\lambda x)$ for $\lambda>0$. Then
    \[\Ent_1(f_{\lambda})=\Ent_1(f)+\log\lambda.\]
\end{prop}

\begin{proof}
    We first have $\|f_{\lambda}\|_1=\lambda^{-n}\|f\|_1$ and
    \begin{equation}\label{fclogfc}
        \int_{\rn}f_{\lambda}(x)\log f_{\lambda}(x)dx=\lambda^{-n}\int_{\rn}f(x)\log f(x)dx.
    \end{equation}

    %\begin{equation}\label{fc=f}
    %    \|f_{\lambda}\|_1=\int_{\rn}f(\lambda x)dx=\lambda^{-n}\|f\|_1
    %\end{equation}
    %and

    Note that $\{f_\lambda\ge t\}=\lambda^{-1}\{f\ge t\}$. Thus
    \[\frac{d\mu_{f_{\lambda}}(t)}{dt}=\frac{|\{f\ge t\}|}{\|f\|_1}=\frac{d\mu_f(t)}{dt}.\]
    Therefore
    \begin{equation}\label{E1step}
        \int_0^{\|f_{\lambda}\|_\infty}\Ent_1(\{f_\lambda\ge t\})d\mu_{f_{\lambda}}(t)=\int_0^{\|f\|_\infty}\Ent_1(\lambda^{-1}\{f\ge t\})d\mu_f(t).
    \end{equation}

    By the definition of $E_1(K)$, we have
    \begin{equation*}
        \begin{aligned}
            \Ent_1(\lambda^{-1}K)&=-\frac{1}{|\lambda^{-1}K|}\int_{\Aff}|\lambda^{-1}K\cap l|\log |\lambda^{-1}K\cap l|dl\\
            &=-\frac{\lambda^n}{n\omega_n|K|}\int_{\sn}\int_{\lambda^{-1}K|u^{\perp}}X_{\lambda^{-1}K}(z,u)\log X_{\lambda^{-1}K}(z,u)d\tH^{n-1}(z)du.
        \end{aligned}
    \end{equation*}
    Since $X_{\lambda^{-1}K}(\lambda^{-1}y,u)=\lambda^{-1}X_K(y,u)$ for $y\in K|u^{\perp}$, we obtain
    \begin{equation*}
        \begin{aligned}
            \Ent_1(\lambda^{-1}K)=&-\frac{\lambda^n}{n\omega_n|K|}\int_{\sn}\int_{K|u^{\perp}}\lambda^{-n}X_K(y,u)\log X_K(y,u)d\tH^{n-1}(y)du\\
            &\qquad\qquad\qquad\qquad\qquad-\frac{\lambda^n\log\lambda^{-1}}{n\omega_n|K|}\int_{\sn}\int_{K|u^{\perp}}\lambda^{-n}X_K(y,u)d\tH^{n-1}(y)du.
           % &=-\frac{1}{n\omega_n|K|}\int_{\sn}\int_{K|u^{\perp}}X_K(y,u)\log X_K(y,u)d\tH^{n-1}(y)du+\log\lambda\\
           %&=\Ent_1(K)+\log\lambda,
        \end{aligned}
    \end{equation*}
    Note that  
    $|K|=\int_{K|u^{\perp}}X_K(y,u)d\tH^{n-1}(y)$, which yields 
    \begin{equation}\label{E1cK}
        \Ent_1(\lambda^{-1}K)=\Ent_1(K)+\log\lambda.
    \end{equation}

    Combining $\|f_{\lambda}\|_1=\lambda^{-n}\|f\|_1$, \eqref{fclogfc}, \eqref{E1step} and \eqref{E1cK}, we have
    \begin{equation*}
        \begin{aligned}
            \Ent_1(f_{\lambda})&=\int_0^{\|f_{\lambda}\|_{\infty}}\Ent_1(\lambda^{-1}\{f\ge t\})d\mu_f(t)-\frac{1}{n\|f_{\lambda}\|_1}\int_{\rn}f_{\lambda}(x)\log f_{\lambda}(x)dx
            %\\&=\int_0^{\|f\|_{\infty}}(\Ent_1(\{f\ge t\})+\log\lambda)d\mu_f(t)-\frac{1}{n\|f\|_1}\int_{\rn}f(x)\log f(x)dx
            =\Ent_1(f)+\log\lambda,
        \end{aligned}
    \end{equation*}
    which concludes the proof.
\end{proof}

We are now in a position to prove the following chord entropy inequality for $\Ent_1(f)$, where $B_f$ denotes the ball with $|B_f|=\|f\|_1$.

\begin{thm}\label{chord-entropy 1}
    Let $f\in L^1(\rn)$ be a non-zero, log-concave function. Then
    \begin{equation}\label{Ent1-proof}
        \Ent_1(f)\ge \Ent_1(B_f)=\Ent_1(B^n)+\frac{1}{n}\log\Big(\frac{\omega_n}{\|f\|_1}\Big),
    \end{equation}
    with equality if and only if $f$ is a constant multiple of the characteristic function of a ball.
\end{thm}

\begin{proof}
We first prove the second equality. For a log-concave $f\in L^1(\rn)$, let $B_f=rB^n$ with $r=(\|f\|_1/\omega_n)^{1/n}$. By Proposition \ref{E1cf}, we have
\[\Ent_1(B_f)=\Ent_1(B^n)+\frac{1}{n}\log\Big(\frac{\omega_n}{\|f\|_1}\Big).\]

For the first inequality, recall that $|\{f^{\star}\ge t\}|=|\{f\ge t\}|$ and $\|f^{\star}\|_p=\|f\|_p$ for all $p>0$. Hence $\mu_f=\mu_{f^{\star}}$, and 
   \[\int_{\rn}f^{\star}(x)\log f^{\star}(x)dx=\int_{\rn}f(x)\log f(x)dx.\]
   By the geometric chord entropy inequality \eqref{1}, we have
    \begin{equation*}
        \begin{aligned}
            \Ent_1(f)&=\int_{0}^{\infty}\Ent_1(\{f\ge t\})d\mu_f(t)-\frac{1}{n\|f\|_1}\int_{\rn}f(x)\log f(x)dx\\
            &\ge \int_0^{\infty}\Ent_1(\{f\ge t\}^{\star})d\mu_{f^{\star}}(t)-\frac{1}{n\|f^{\star}\|_1}\int_{\rn}f^{\star}(x)\log f^{\star}(x)dx=\Ent_1(f^{\star}),
        \end{aligned}
    \end{equation*}
    where equality holds if and only if superlevel sets $\{f\ge t\}$ are balls for almost all $t>0$.

    Proposition \ref{E1cf} implies that
    \begin{equation*}
        \begin{aligned}
            \Ent_1(f^{\star})&=\int_0^{\infty}\Big(\Ent_1(B^n)+\frac{1}{n}\log\Big(\frac{\omega_n}{|\{f\ge t\}|}\Big)\Big)d\mu_{f^{\star}}(t)-\frac{1}{n\|f^{\star}\|_1}\int_{\rn}f^{\star}(x)\log f^{\star}(x)dx\\
            &=\Ent_1(B^n)+\frac{1}{n}\log\Big(\frac{\omega_n}{\|f\|_1}\Big)+q(f^{\star}),
        \end{aligned}
    \end{equation*}
    where the functional $q(f)$ is given by
    \[q(f)=\frac{1}{n}\log\|f\|_1-\frac{1}{n\|f\|_1}\int_0^{\infty}|\{f\ge t \}|\log |\{f\ge t\}|dt-\frac{1}{n\|f\|_1}\int_{\rn}f(x)\log f(x)dx.\]
    Note that $q(f)=q(f^{\star})$. It remains to prove $q(f)\ge 0$ and the equality characterization.

    Let $p(x,t):\rn\times \tr^+\rightarrow (0,\infty)$ be defined by
    \[p(x,t)=\frac{\chi_{\{f\ge t\}}(x)}{\|f\|_1},\]
    which is a probability density function on $\rn\times\tr^{+}$. Its marginals are the following two probability density functions:
    \[\bar{f}(x)=\frac{f(x)}{\|f\|_1}~~\text{and}~~s(t)=\frac{|\{f\ge t\}|}{\|f\|_1}.\]
    We now apply Lemma \ref{rel-ent} with $h=p$ and $l=\bar{f}s$. Then we obtain
    \begin{equation}\label{app-rel ent}
        \begin{aligned}
            \int_{\rn}\int_0^{\infty}p(x,t)\log\Big(\frac{p(x,t)}{\bar{f}(x) s(t)}\Big)dtdx\ge 0,
        \end{aligned}
    \end{equation}
    where 
    \begin{equation*}
        \begin{aligned}
            \int_{\rn}&\int_0^{\infty}p(x,t)\log\Big(\frac{p(x,t)}{l(x, t)}\Big)dtdx\\
            &=-\log\|f\|_1-\frac{1}{\|f\|_1}\int_{\rn}f(x)\log \Big(\frac{f(x)}{\|f\|_1}\Big)dx-\frac{1}{\|f\|_1}\int_0^{\infty}|\{f\ge t\}|\log\Big(\frac{|\{f\ge t\}|}{\|f\|_1}\Big)dt\\
            &=nq(f).
        \end{aligned}
    \end{equation*}

    Therefore \eqref{app-rel ent} implies that $q(f)\ge 0$ and thus
    \begin{equation*}
        \Ent_1(f)\ge \Ent_1({B^n})+\frac{1}{n}\log\Big(\frac{\omega_n}{\|f\|_1}\Big).
    \end{equation*}
    As discussed above, equality holds if and only if superlevel sets $\{f\ge t\}$ are balls for almost all $t>0$ and $q(f)=0$. By Lemma \ref{rel-ent}, equality is attained in \eqref{app-rel ent} if and only if $p(x,t)=\bar{f}(x)s(t)$ for almost all $x\in\rn$ and $t>0$. Hence $f=c\chi_{B}$, where $c>0$ and $B$ is a ball.
\end{proof}

\begin{rem}\label{remE1}
In the proof of Theorem~\ref{chord-entropy 1}, the convexity of the superlevel sets $\{f\ge t\}$ is the main geometric property of log-concave functions that is used. Consequently, the theorem extends to quasi-concave functions, provided the relevant integrability assumptions are satisfied. Here $f$ is quasi-concave if $\{f\ge t\}$ is convex for every $t\in\mathbb{R}$.
\end{rem}

For a convex, bounded set $K\subset\rn$ with non-empty interior, its closure $\bar K$ is a convex body. Hence
\begin{equation}\label{E1-section}
\begin{aligned}
        \Ent_1(\bar K)=-\frac{1}{n\omega_n|K|}\int_{\sn}\int_{K|u^{\perp}}X_{\bar K}(y,u)\log X_{\bar K}(y,u)d\tH^{n-1}(y)du.
\end{aligned}
\end{equation}
Since for almost all $u\in\sn$, the intersection $\partial K\cap(y+\tr u)$ only contains two points, 
\[|K\cap(y+\tr u)|=|\bar K\cap (y+\tr u)|,\]
for $\tH^{n-1}$-almost all $y\in \bar K|u^{\perp}$ and $u\in\sn$. Therefore, if $K\subset \rn$ is a convex, bounded set with non-empty interior, we can extend the definition \eqref{E1-section} by letting $\Ent_1(K):=\Ent_1(\bar K)$ and then get
\begin{equation}\label{ex-EI1}
    \Ent_1(K)\ge \Ent_1(B_K)=\Ent_1(B^n)+\frac{1}{n}\log\Big(\frac{\omega_n}{|K|}\Big).
\end{equation}

By \eqref{ex-EI1}, Lemma \ref{ex-Xu3.2} and Remark \ref{remE1}, we have the following corollary, which extends Theorem \ref{chord-entropy 1} to quasi-concave functions. 

\begin{cor}\label{ex-C1}
    Let $\beta\in(0,1)$ be fixed, and $f\in L^1(\rn)$ be a non-negative, non-zero quasi-concave function such that $\Ent_1(f)$ is finite. Then
    \[\Ent_1(f)\ge \Ent_1(B_f)=\Ent_1(B^n)+\frac{1}{n}\log\Big(\frac{\omega_n}{\|f\|_1}\Big),\]
    with equality if and only if $f$ is a constant multiple of the characteristic function of a ball.
\end{cor}

%%%%%%%%%%%%%%%%%%%%%%% The following are for n+1%%%%%%%%%%%%%%%%%%%%%%%%%%%%%%%%%%%
\subsection{The chord entropy inequality for \texorpdfstring{$\Ent_{n+1}(f)$}{}}\hfill

With the limiting case $\alpha\rightarrow 0^+$ fully established, we turn to the case $\alpha\rightarrow n$, which exhibits fundamentally different behavior. When $\alpha=n$, Theorem A and Theorem B imply that

\begin{equation}\label{chord Sob-n}
    \frac{\|f\|_1^2}{\|f\|_{\infty}}\leq \di \min\{f(x), f(y)\}dxdy\leq \|f\|_{\frac{1}{2}}
\end{equation}
for log-concave functions $f\in L^{1}(\rn)$. Therefore the chord Sobolev inequality does not reduce to an identity when $\alpha=n$. Even so, inequalities in \eqref{chord Sob-n} are  trivial, which follow from
\[\frac{f(x)f(y)}{\|f\|_{\infty}}\leq  \min\{f(x), f(y)\},\qquad\min\{f(x), f(y)\}\leq f(x)^{\frac{1}{2}}f(y)^{\frac{1}{2}}.\]

In this case, the definition of $\Ent_{n+1}(f)$ is primarily motivated by the geometric quantities  $\Ent_{n+1}(K)$. Recall that for a convex body $K\in\rn$, 
\[\Ent_{n+1}(K)=-\frac{n+1}{I_{n+1}(K)}\int_{\Aff}|K\cap l|^{n+1}\log |K\cap l|dl,\]
where $I_{n+1}(K)=\frac{n+1}{\omega_n}|K|^2$ by the Poincar\'e-Hadwiger's integral formula \eqref{CCPH}.

For a log-concave $f\in L^1(\rn)$, we define
\[\Ent_{n+1}(f)=\int_0^{\infty}\Ent_{n+1}(\{f\ge t\})d\nu_f(t),\]
where
\[d\nu_f(t)=\frac{(n+1)|\{f\ge t\}|^2}{\omega_nI_{n+1}(f)}dt=\frac{I_{n+1}(\{f\ge t\})}{I_{n+1}(f)}dt\]
is a probability measure on $\rn$. Since $s\mapsto -s^{n+1}\log s$ is bounded from above, $\Ent_{n+1}(f)$ admits an upper bound. This ensures that $\Ent_{n+1}(f)$ is well-defined.

\begin{prop}
Let $f \in L^{1}(\rn)$ be non-zero, log-concave, and $f_{\lambda}(x)=f(\lambda x)$ for $\lambda>0$. Then
\[
\Ent_{n+1}(f_{\lambda})
=
\Ent_{n+1}(f)
+
(n+1)\log\lambda.
\]
\end{prop}

\begin{proof}
    Recall that $\{f_\lambda\ge t\}=\lambda^{-1}\{f\ge t\}$ and hence
    \[I_{n+1}(\{f_{\lambda}\ge t\})=\frac{n+1}{\omega_n}|\{f_{\lambda}\ge t\}|^2=\lambda^{-2n}I_{n+1}(\{f\ge t\}).\]
    Moreover, note that
    \[I_{n+1}(f_\lambda)=\frac{n+1}{\omega_n}\di \min\{f_\lambda(x),f_\lambda(y)\}dxdy=\frac{n+1}{\omega_n}\int_0^{\|f\|_{\infty}}|\{f_\lambda\ge t\}|^2dt=\lambda^{-2n}I_{n+1}(f),\]
    which implies that $\nu_{f_\lambda}=\nu_f$.

    By the definition of $\Ent_{n+1}(K)$ and the identity $X_{\lambda^{-1}K}(\lambda^{-1}y,u)=\lambda^{-1}X_K(y,u)$,
    \begin{equation}\label{En+1cK}
        \begin{aligned}
            \Ent_{n+1}(\lambda^{-1}K)
            &=-\frac{n+1}{\lambda^{-2n}n\omega_nI_{n+1}(K)}\int_{\sn}\int_{\lambda^{-1}K|u^{\perp}}X_{\lambda^{-1}K}(z,u)^{n+1}\log X_{\lambda^{-1}K}(z,u)d\tH^{n-1}(z)du\\
            &=-\frac{n+1}{n\omega_nI_{n+1}(K)}\int_{\sn}\int_{K|u^{\perp}}X_K(y,u)^{n+1}\log X_K(y,u)d\tH^{n-1}(y)du\\
            &\qquad+\frac{n+1}{n\omega_nI_{n+1}(K)}\log\lambda\int_{\sn}\int_{K|u^{\perp}}X_K(y,u)^{n+1}d\tH^{n-1}(y)du\\
            &=\Ent_{n+1}(K)+(n+1)\log\lambda.
        \end{aligned}
    \end{equation}
    Therefore
    \[\Ent_{n+1}(f_{\lambda})=\int_0^{\|f\|_{\infty}}\Ent_{n+1}(\lambda^{-1}\{f\ge t\})d\nu_f(t)=\Ent_{n+1}(f)+(n+1)\log\lambda,\]
    which completes the proof.
\end{proof}

\begin{thm}
    Let $f\in L^1(\rn)$ be a non-zero, log-concave function. Then
    \[\Ent_{n+1}(f)\leq \Ent_{n+1}(\tilde{B}_f)=\Ent_{n+1}(B^n)+\frac{n+1}{n}\log\Big(\frac{(n+1)\|f\|_1}{I_{n+1}(f)}\Big),\]
    where $\tilde{B_f}$ is a ball with $|\tilde{B}_f|=\frac{\omega_n I_{n+1}(f)}{(n+1)\|f\|_1}$. Equality holds if and only if $f$ is a constant multiple of the characteristic function of a ball in $\rn$.
\end{thm}

\begin{proof}

By \eqref{En+1cK}, we have the second equality.
%\begin{equation*}
%        \Ent_{n+1}(\tilde{B}_f)=\Ent_{n+1}(B^n)+\frac{n+1}{n}\log\Big(\frac{\omega_n}{|\tilde{B}_f|}\Big).
%\end{equation*}
%Since $|\tilde{B}_f|=\frac{ \omega_n I_{n+1}(f)}{(n+1)\|f\|_1}$, we obtain the second equality. 
    Recall that $\Ent_{n+1}(K)\leq \Ent_{n+1}(B_K)$, where $B_K$ is the centered ball with the same volume as $K$. Since $\nu_f=\nu_{f^{\star}}$, it follows from the geometric chord entropy inequality \eqref{n+1} that
    \begin{equation*}
        \begin{aligned}
            \Ent_{n+1}(f)&=\int_{0}^{\|f\|_{\infty}}\Ent_{n+1}(\{f\ge t \})d\nu_f(t)\\
            &\leq \int_0^{\|f\|_{\infty}}\Ent_{n+1}(\{f^{\star}\ge t\})d\nu_{f^{\star}}(t)\\
            &=\int_{0}^{\|f\|_{\infty}}\Big(\Ent_{n+1}(B^n)+\frac{n+1}{n}\log\Big(\frac{\omega_n}{|\{f^{\star}\ge t\}|}\Big)\Big)d\nu_{f^{\star}}(t)\\
            &=\Ent_{n+1}(B^n)+\frac{n+1}{n}\log\omega_n-\frac{(n+1)^2}{n\omega_nI_{n+1}(f)}\int_0^{\|f\|_{\infty}}|\{f\ge t\}|^2\log|\{f\ge t\}|dt,
        \end{aligned}
    \end{equation*}
    where the last equality is from the fact that $|\{f^{\star}\ge t\}|=|\{f\ge t\}|$ and $I_{n+1}(f)=I_{n+1}(f^{\star})$. By \eqref{n+1}, equality holds if and only if superlevel sets $\{f\ge t\}$ are centered balls for almost all $t$.
    
    Since $h(t)=t\log t$ is a convex function on $(0,\infty)$, Jensen's inequality implies that
    \begin{equation*}
        \begin{aligned}
            \frac{1}{\|f\|_1}\int_0^{\|f\|_\infty}|\{f\ge t\}|^2\log |\{f\ge t\}|dt&=\int_0^{\|f\|_{\infty}}h(|\{f\ge t\}|)d\mu_f(t)\\
            &\ge h\Big(\int_0^{\|f\|_\infty}|\{f\ge t\}|d\mu_f(t)\Big)\\
            &=\frac{\omega_nI_{n+1}(f)}{(n+1)\|f\|_1}\log\Big(\frac{\omega_nI_{n+1}(f)}{(n+1)\|f\|_1}\Big),
        \end{aligned}
    \end{equation*}
    where equality holds if and only if $|\{f\ge t\}|$ is a constant. This implies that $f=c\chi_E$ where $c>0$ and $E\subset\rn$ is a measurable set.
    
    Hence we have
    \begin{equation}\label{En+1f-step}
        \begin{aligned}
            \Ent_{n+1}(f)&\leq \Ent_{n+1}(B^n)+\frac{n+1}{n}\log\omega_n-\frac{n+1}{n}\log\Big(\frac{\omega_nI_{n+1}(f)}{(n+1)\|f\|_1}\Big)\\
            &=\Ent_{n+1}(B^n)+\frac{n+1}{n}\log\Big(\frac{(n+1) \|f\|_1}{I_{n+1}(f)}\Big),
        \end{aligned}
    \end{equation}
    and equality holds if and only if superlevel sets $\{f\ge t\}$ are balls for almost all $t\in[0,\|f\|_\infty]$ and $|\{f\ge t\}|$ is a constant.  Therefore equality holds in \eqref{En+1f-step} if and only if $f=c\chi_{B}$ with $c>0$ and $B\subset\rn$ a centered ball such that $c|B|=\|f\|_1$.
\end{proof}

%{\cb According to the proof, maybe the coefficient in the definition of $\Ent_{n+1}(K)$ could be simplified. Deleting the n+1 will simplify the notation and simplify the result in the following way:
%\[\Ent_{n+1}(f)\leq \Ent_{n+1}(B^n)+\frac{1}{n}\log\Big(\frac{(n+1)\|f\|_1}{I_{n+1}(f)}\Big)\]}

\section{A logarithmic Sobolev-type inequality}\label{log Sob ineq}

This section is devoted to the proof of Theorem C, which extends the chord
entropy inequality \eqref{Ent1-proof} from the log-concave and
quasi-concave settings to a broader fractional Sobolev class. We first unfold
the definition of \(\Ent_1(f)\) and rewrite \eqref{Ent1-proof} in the
following form:
\begin{equation}\label{limiting-0}
    \bar\sigma_0+\frac{1}{n}\log\|f\|_1
    -\frac{1}{n\|f\|_1}\int_{\rn}f(x)\log f(x)\,dx
    \geq
    \frac{1}{n\omega_n}\int_{\sn}\log\rho_{\rR_0f}(u)\,du .
\end{equation}
This reformulation follows from Lemma \ref{ex-Xu3.2}. The sharp constant is
given by
\begin{equation*}
    \bar\sigma_0
    =
    \frac{1}{n\omega_n}
    \frac{d}{d\alpha}\Big|_{\alpha=0}\alpha\sigma_{n,\alpha}
    =
    -\Ent_1(B^n)-\frac{1}{n}\log\omega_n-1.
\end{equation*}
Recall also that, by Corollary \ref{ex-C1}, inequality \eqref{limiting-0}
holds for every non-negative, non-zero quasi-concave function
\(f\in L^1(\rn)\) such that \(f\log f\in L^1(\rn)\) and
\(\log\rho_{\rR_0f}\in L^1(\sn)\).

We first prove \eqref{limiting-0} for non-negative, non-zero functions
\(f\in W^{\beta,1}(\rn)\) with compact support and satisfying \(f\log f\in L^1(\rn)\). The main step is the following rearrangement inequality:
\begin{equation}\label{vlog-rri}
    \tilde{V}_{\log}(B^n,\rR_0 f^{\star})\geq \tilde{V}_{\log}(B^n,\rR_0f),
\end{equation}
where
\[\tilde{V}_{\log}(B^n,\rR_0f):=\frac{1}{n\omega_n}\int_{\sn}\log\rho_{\rR_0f}(u)du.\]

Assuming \eqref{vlog-rri}, we proceed as follows. By the P\'olya--Szeg\H{o} inequality established in
\cite{AL},
\[\di\frac{|f^{\star}(x)-f^{\star}(y)|}{|x-y|^{n+\beta}}dxdy\leq \di\frac{|f(x)-f(y)|}{|x-y|^{n+\beta}}dxdy<\infty.\]
Hence \(f^{\star}\in W^{\beta,1}(\rn)\). Since \(f\) has compact support,
so does \(f^\star\). Therefore $f^{\star}\in L^{\frac12}(\rn)\cap W^{\beta,1}(\rn)$ and Lemma \ref{ex-R0f} gives $\log\rho_{\rR_0f^{\star}}\in L^1(\sn)$. 

Since \(f^\star\) is quasi-concave, and symmetric decreasing rearrangement
preserves both the \(L^1\)-norm and the integral of \(f\log f\), we may apply
\eqref{limiting-0} to \(f^\star\). Combining the resulting inequality with
\eqref{vlog-rri}, we obtain
\begin{equation}\label{app}
    \bar\sigma_0+\frac{1}{n}\log\|f\|_1-\frac{1}{n\|f\|_1}\int_{\rn}f(x)\log f(x)dx\ge\tilde{V}_{\log}(B^n,\rR_0 f^{\star})\geq \tilde{V}_{\log}(B^n,\rR_0f).
\end{equation}

We shall first prove rearrangement inequalities for two auxiliary truncated
quantities, which will then be combined to obtain \eqref{vlog-rri}. Related results may be found in \cite{Cai, HL1}. 

%{\cb Recall that, in Section \ref{rdm-fcn}, we introduced
%\[\mathcal{E}_u(f)=\int_0^1\frac{1}{r}\int_{\rn} |f(x)-f(x+ru)|dxdr\]
%and
%\[\mathcal{Q}_u(f)=\int_1^{\infty}\frac{1}{r}\int_{\rn} \min\{f(x), f(x+ru)\}dxdr.\]
%We shall prove rearrangement inequalities for
%\(\int_{\sn}\mathcal{E}_u(f)\,du\) and
%\(\int_{\sn}\mathcal{Q}_u(f)\,du\), respectively.}

\begin{lem}\label{P-S01}
Let $f\in L^1(\rn)$ be a non-negative, non-zero function, and
\begin{equation}\label{Ef}
    \mathcal{E}_1(f):= \di |f(x)-f(y)|\max\{0,|x-y|^{-n}-1\}dxdy.
\end{equation}
If $\mathcal{E}_1(f)<\infty$, then
\[\mathcal{E}_1(f^\star)\leq \mathcal{E}_1(f).\]
Equality holds if and only if the superlevel sets \(\{f\ge t\}\) are balls
for almost all \(t>0\), up to null sets.
\end{lem}

\begin{proof}
    Since 
    $|z|^{-n}=\int_0^{\infty}\chi_{r^{-\frac{1}{n}}B^n}(z)dr$, the kernel function is
    \[\max\{0, |x-y|^{-n}-1\}=\int_1^{\infty}\chi_{r^{-\frac{1}{n}}B^n}(x-y)dr.\]
    Therefore
    \[\mathcal{E}_1(f)=\int_1^{\infty}\di |f(x)-f(y)|\chi_{r^{-\frac{1}{n}}B^n}(x-y)dxdydr.\]
    Moreover, by 
    \begin{equation*}
        |f(x)-f(y)|=\int_0^{\infty}|\chi_{ \{f\ge t\} }(x)- \chi_{ \{f\ge t\} }(y)|dt,
    \end{equation*}
    we have
    \begin{equation*}
        \begin{aligned}
        \mathcal{E}_1(f)&=\int_1^{\infty}\int_0^{\infty}\di|\chi_{ \{f\ge t\} }(x)-\chi_{ \{f\ge t\} }(y) |\chi_{r^{-\frac{1}{n}}B^n}(x-y)dxdydtdr\\
            &=2\int_1^{\infty}\int_0^{\infty}\di \Big(\chi_{ \{f\ge t\} }(x)-\chi_{ \{f\ge t\} }(x)\chi_{ \{f\ge t\} }(y)\Big)\chi_{r^{-\frac{1}{n}}B^n}(x-y)dxdydtdr\\
            &=2\int_1^{\infty}\int_0^{\infty}\Big(r^{-1}\omega_n|\{f\ge t\}|-\di\chi_{ \{f\ge t\} }(x)\chi_{r^{-\frac{1}{n}}B^n}(x-y)\chi_{\{f\ge t\} }(y)dxdy\Big)dtdr.
        \end{aligned}
    \end{equation*}

    By the Riesz rearrangement inequality in Theorem \ref{rri-cha}, we have
    \begin{equation}\label{rri s}
        \begin{aligned}
        \di \chi_{\{f\ge t\}}(x)\chi_{r^{-\frac{1}{n}}B^n}&(x-y)\chi_{\{f\ge t\}}(y)dxdy\\
        &\leq \di \chi_{\{f^{\star}\ge t\}}\chi_{r^{-\frac{1}{n}}B^n}(x-y)\chi_{\{f^{\star}\ge t\}}(y)dxdy.
    \end{aligned}
    \end{equation}
    Since $|\{f^{\star}\ge t\}|=|\{f\ge t\}|$, it follows that
    \begin{equation*}
        \begin{aligned}
            \mathcal{E}_1(f)&\ge2\int_1^{\infty}\int_0^{\infty}\Big(r^{-1}\omega_n|\{f^{\star}\ge t\}|-\di\chi_{ \{f^{\star}\ge t\} }(x)\chi_{r^{-\frac{1}{n}}B^n}(x-y)\chi_{\{f^{\star}\ge t\} }(y)dxdy\Big)dtdr\\
            &=\int_1^{\infty}\di|f^{\star}(x)-f^{\star}(y)|\chi_{r^{-\frac{1}{n}}B^n}(x-y)dxdy=\mathcal{E}_1(f^{\star}).
        \end{aligned}
    \end{equation*}
    
    If equality holds, then equality holds in \eqref{rri s} for almost all $(r,t)\in(1,\infty)\times(0,\infty)$. For such $(r,t)$ with $r>1$ sufficiently large, the assumptions of Theorem \ref{rri-cha} are satisfied, and hence
    \[\{f\ge t\}=x_0+a B,~~~~ r^{-\frac{1}{n}}B^n=b B\]
    where $B$ is a centered ball, $x_0\in\rn$, and $a,b>0$. This concludes the proof.
\end{proof}

By a similar application of the Riesz rearrangement inequality, we obtain the following lemma.

\begin{lem}\label{P-S1infty}
    Let $f\in L^1(\rn)$ be a non-negative, non-zero function, and
\begin{equation}\label{Qf}
    \mathcal{Q}_1(f):= \di \min\{f(x), f(y)\}\min\{1,|x-y|^{-n}\}dxdy.
\end{equation}
If $\mathcal{Q}_1(f)<\infty$, then $\mathcal{Q}_1(f)\leq \mathcal{Q}_1(f^\star)$.
\end{lem}

\begin{proof}
    We first rewrite the kernel function as
    \[\min\{1,|x-y|^{-n}\}=\int_0^{1}\chi_{r^{-\frac{1}{n}}B^n}(x-y)dr,\]
    which yields
    \begin{equation*}
        \begin{aligned}
            \mathcal{Q}_1(f)=\int_0^1\di\min\{f(x),f(y)\}\chi_{r^{-\frac{1}{n}}B^n}(x-y)dxdydr.
        \end{aligned}
    \end{equation*}
    By the layer cake representation of $\min\{f(x),f(y)\}$ and Theorem \ref{rri-cha}, we have
    \begin{equation*}
        \begin{aligned}
            \mathcal{Q}_1(f)
            &=\int_0^1\int_0^{\infty}\di\chi_{ \{f\ge t\} }(x)\chi_{r^{-\frac{1}{n}}B^n}(x-y)\chi_{ \{f\ge t\} }(y)dxdydtdr\\
            &\leq\int_0^1\int_0^{\infty}\di\chi_{ \{f^{\star}\ge t\} }(x)\chi_{r^{-\frac{1}{n}}B^n}(x-y)\chi_{ \{f^{\star}\ge t\} }(y)dxdydtdr\\
            &=\int_0^1\di\min\{f^{\star}(x),f^{\star}(y)\}\chi_{r^{-\frac{1}{n}}B^n}(x-y)dxdydr=\mathcal{Q}_1(f^{\star}),
        \end{aligned}
    \end{equation*}
    which concludes the proof.
    %If equality holds, for almost all $(r,t)\in (0,1)\times (0,\infty)$, there is
    %\begin{align*}
    %    \di\chi_{ \{f\ge t\} }(x)&\chi_{r^{-\frac{1}{n}}B^n}(x-y)\chi_{ \{f\ge t\} }(y)dxdydtdr\\
    %        &=\di\chi_{ \{f^{\star}\ge t\} }(x)\chi_{r^{-\frac{1}{n}}B^n}(x-y)\chi_{ \{f^{\star}\ge t\} }(y)dxdydtdr.
    %\end{align*}
    %Then for such $(r,t)$ with $r>1$ sufficiently large, the assumptions of Theorem \ref{rri-cha} are fulfilled and thus
    %\[\{f\ge t\}=x_0+a B,~~~~ r^{-\frac{1}{n}}B^n=b B\]
    %where $B$ is a centered ball and $x_0\in\rn$ and $a,b>0$, which concludes the proof.
\end{proof}

We are now ready to prove \eqref{vlog-rri}. Recall that, in Section \ref{rdm-fcn}, we defined
\[\mathcal{E}_u(f)=\int_0^1\frac{1}{r}\int_{\rn}|f(x)-f(x+ru)|dxdr\]
and
\[\mathcal{Q}_u(f)=\int_1^\infty\frac{1}{r}\int_{\rn}\min\{f(x),f(x+ru)\}dxdr. \]
\begin{thm}\label{vlog-rri-thm}
    Let $\beta\in(0,1)$ be fixed, and $f\in W^{\beta,1}(\rn)$ be a non-negative, non-zero function. Then
    \[\tilde{V}_{\log}(B^n, \rR_0f^{\star})\ge \tilde{V}_{\log}(B^n, \rR_0f),\]
    whenever the right-hand side is finite. Equality holds if and only if the superlevel sets \(\{f\ge t\}\) are balls for almost all \(t>0\), up to null sets.
\end{thm}

\begin{proof}
    By \eqref{0EQ}, there is a constant $C$ such that
    \begin{equation}\label{step1}
        \tilde{V}_{\log}(B^n,\rR_0f)=C-\frac{1}{2n\omega_n\|f\|_1}\int_{\sn}\mathcal{E}_u(f)du+\frac{1}{n\omega_n\|f\|_1}\int_{\sn}\mathcal{Q}_u(f)du.
    \end{equation}
    By polar coordinates, the definition of $\mathcal{E}_u(f)$ and the definition of $\mathcal{E}_1(f)$ given in \eqref{Ef},
    \begin{equation}\label{step2}
        \int_{\sn}\mathcal{E}_u(f)du=\mathcal{E}_1(f)+\di |f(x)-f(y)|\chi_{B^n}(x-y)dxdy.
    \end{equation}
    Similarly, by polar coordinates, the definition of $\mathcal{Q}_u(f)$ and the definition of $\mathcal{Q}_1(f)$ given in \eqref{Qf},
    \begin{equation}\label{step3}
        \int_{\sn}\mathcal{Q}_u(f)du=\mathcal{Q}_1(f)-\di \min\{f(x),f(y)\}\chi_{B^n}(x-y)dxdy.
    \end{equation}

    Since
    \[\frac{1}{2}|f(x)-f(y)|+\min\{f(x), f(y)\}=\frac{1}{2}(f(x)+f(y)),\]
    by \eqref{step1}, \eqref{step2} and \eqref{step3}, we obtain 
    \[\tilde{V}_{\log}(B^n,\rR_0f)=C-\frac{1}{2n\omega_n\|f\|_1}\mathcal{E}_1(f)+\frac{1}{n\omega_n\|f\|_1}\mathcal{Q}_1(f)-\frac{1}{n}.\]
    Since $\|f\|_1=\|f^{\star}\|_1$, Lemma \ref{P-S01} and Lemma \ref{P-S1infty} imply that
    \[\tilde{V}_{\log}(B^n,\rR_0f^\star)\ge \tilde{V}_{\log}(B^n,\rR_0f).\]
    If equality holds in the above inequality, then $\mathcal{E}_1(f)=\mathcal{E}_1(f^\star)$. Lemma 8.1 thereby implies that the superlevel sets \(\{f\ge t\}\) are balls
for almost all \(t>0\), up to null sets.
\end{proof}

By Theorem \ref{vlog-rri-thm} and Corollary \ref{ex-C1}, we obtain the following lemma.

\begin{lem}\label{C-compact support}
    Let $\beta\in (0,1)$ be fixed, and $f\in W^{\beta,1}(\rn)$ be a non-negative, non-zero function with compact support. Suppose $f\log f\in L^1(\rn)$. Then
    \[\bar\sigma_{0}+\frac{1}{n}\log\|f\|_1-\frac{1}{n\|f\|_1}\int_{\rn}f(x)\log f(x)dx\ge \frac{1}{n\omega_n}\int_{\sn}\log\rho_{\rR_0f}(u)du.\]
\end{lem}

 \begin{proof}
    It follows from Theorem \ref{vlog-rri-thm} and \eqref{app}.
\end{proof}

To remove the compact-support assumption, we shall use an approximation
argument based on smooth cut-off functions. We first record the required
technical lemma; see also \cite{Cai}.

\begin{lem}\label{cuf-off}
    Let $\beta\in (0,1)$ be fixed and $f\in W^{\beta,1}(\rn)$ be a non-negative, non-zero function. Let $\eta\in C_c^\infty(\rn)$ be a cut-off function such that
$0\le \eta\le1$, with $\eta=1$ on $B^n$ and $\eta=0$ outside $2B^n$. Set
\[\eta_j(x)=\eta(x/j),\qquad f_j=\eta_j f .\]
Then $f_j\in W^{\beta, 1}(\rn)\cap L^{1/2}(\rn)$, and
\begin{equation}\label{mono2}
    \lim_{j\to\infty}\int_{\sn}\log\rho_{\rR_0 f_j}(u)du=\int_{\sn}\log\rho_{\rR_0 f}(u)du .
\end{equation}
\end{lem}

\begin{proof}
Since $\eta_j\in C_c^\infty(\rn)$ and
$f\in W^{\beta,1}(\rn)$, the multiplication property in
\cite[Theorem 6.23]{Leoni} gives $f_j\in W^{\beta,1}(\rn)$. Moreover, $f_j$ has compact support, indicating that $f_j\in L^{1/2}(\rn)$.

    We still use $\mathcal{E}_u(f)$ and $\mathcal{Q}_u(f)$ defined above. Recall that
    \[\int_{\sn}\log\rho_{\rR_0f}(u)du=n\omega_nC-\frac{1}{2\|f\|_1}\int_{\sn}\mathcal{E}_u(f)du+\frac{1}{\|f\|_1}\int_{\sn}\mathcal{Q}_u(f)du.\]
    Since $0\leq f_j\leq f$ and $f_j\rightarrow f$ pointwise, the dominated convergence theorem implies that $\|f_j\|_1\rightarrow \|f\|_1$. Then it suffices to prove that
    \begin{equation}\label{limE}
        \lim_{j\rightarrow\infty}\int_{\sn}\mathcal{E}_u(f_j)du=\int_{\sn}\mathcal{E}_u(f)du,
    \end{equation}
    and
    \begin{equation}\label{limQ}
        \lim_{j\rightarrow\infty}\int_{\sn}\mathcal{Q}_u(f_j)du=\int_{\sn}\mathcal{Q}_u(f)du.
    \end{equation}

    \noindent{Step 1: We aim to show \eqref{limE}.}

    Note that $|\mathcal{E}_u(f_j)-\mathcal{E}_u(f)|\leq \mathcal{E}_u(f_j-f)$. It suffices to prove that
    \[\int_{\sn}\mathcal{E}_u(f_j-f)du\rightarrow 0.\]
    Let $h_j=f_j-f$. A straightforward calculation shows that
    \[|h_j(x)-h_j(x+ru)|\leq|\eta_j(x)-1|\cdot|f(x)-f(x+ru)|+|\eta_j(x)-\eta_j(x+ru)|f(x+ru).\]
    Then we have
    \[\int_{\sn}\mathcal{E}_u(h_j)du\leq \int_{\sn}{\mathrm I}_j(u)du+\int_{\sn}\mathrm{II}_j(u)du,\]
    where
    \begin{equation*}
            \mathrm I_j(u)=\int_0^1\frac{1}{r}\int_{\rn}|\eta_j(x)-1|\cdot|f(x)-f(x+ru)|dxdr
    \end{equation*}
    and
    \begin{equation*}
        \mathrm{II}_j(u)=\int_0^1\frac{1}{r}\int_{\rn}|\eta_j(x)-\eta_j(x+ru)|f(x+ru)dxdr.
    \end{equation*}

    For $\mathrm{I}_j(u)$, since $|\eta_j(x)-1|\leq 1$ and $\eta_j(x)\rightarrow 1$ pointwise we have
    \[|\eta_j(x)-1|\cdot |f(x)-f(x+ru)|\leq |f(x)-f(x+ru)|,\]
    and the left-hand side converges to $0$ pointwise. By Lemma \ref{E-def}, $\mathcal{E}_u(f)$ is finite for all $u\in\sn$. The dominated convergence theorem implies that, for all $u\in\sn$,
    \[\mathrm{I}_j(u)\rightarrow 0,\qquad \mathrm{I}_j(u)\leq \mathcal{E}_u(f).\]
    By Lemma \ref{E-def}, $\mathcal{E}_u(f)\in L^1(\sn)$. Applying the dominated convergence theorem again, 
    \[\int_{\sn}\mathrm{I}_j(u)du\rightarrow 0.\]

    On the other hand, $\eta\in C_c^{\infty}(\rn)$ implies that there is a constant $M$ such that $\|\nabla\eta\|_{\infty}\leq M$. Therefore $|\eta_j(x)-\eta_j(x+ru)|\leq rM/j$, which yields that
    \[\int_{\sn}\mathrm{II}_j(u)du\leq \int_{\sn}\int_0^1\int_{\rn}\frac{M}{j}f(x+ru)dxdrdu=\frac{n\omega_nM}{j}\|f\|_1\rightarrow 0.\]
    Therefore we complete the proof of \eqref{limE}.
    \vskip 2pt

    \noindent{Step 2: We aim to show \eqref{limQ}.}

    We denote by $\bar f_j(x)=f(x)\chi_{jB^n}(x)$. Since $\eta_j(x)\in [0,1]$ and $\eta_j(x)=1$ for $x\in jB^n$, we have $\bar f_j\leq f_j\leq f$, which implies that
    \[\min\{\bar f_j(x),\bar f_j(x+ru)\}\leq \min\{f_j(x), f_j(x+ru)\}\leq \min\{f(x),f(x+ru)\},\]
    and hence that
    \begin{equation}\label{bound}
        \int_{\sn}\mathcal{Q}_u(\bar f_j)du\leq \int_{\sn}\mathcal{Q}_u(f_j)du\leq \int_{\sn}\mathcal{Q}_u(f)du.
    \end{equation}

    Note that
    \[\int_{\sn}\mathcal{Q}_u(\bar f_j)du=\di\frac{ \min\{\bar f_j(x), \bar f_j(y)\} \chi_{[1,\infty)}(|x-y|)}{|x-y|^n}dxdy.\]
    Applying the monotone convergence theorem twice, we have
    \[\lim_{j\rightarrow\infty}\int_{\sn}\mathcal{Q}_u(\bar f_j)du=\int_{\sn}\mathcal{Q}_u(f)du.\]
    Combined with \eqref{bound}, we conclude the proof of \eqref{limQ}.
\end{proof}

As a consequence of the preceding analysis, we now prove Theorem C.

\begin{proof}[Proof of Theorem C]
    Let $f_j$ be defined in Lemma \ref{cuf-off}. Then
    \[f_j(x)\log f_j(x)=\eta_j(x)\big(f(x)\log f(x)\big)+f(x)\big(\eta_j(x)\log \eta_j(x)\big)\]
    Since $f\log f\in L^1(\rn)$ and
    \[\eta_j(x)f(x)\log f(x)\rightarrow f(x)\log f(x),\qquad |\eta_j(x)f(x)\log f(x)|\leq |f(x)\log f(x)|,\]
    the dominated convergence theorem yields that
    \begin{equation}\label{conv-cut1}
        \lim_{j\rightarrow\infty}\int_{\rn}\eta_j(x)f(x)\log f(x)dx=\int_{\rn}f(x)\log f(x)dx.
    \end{equation}
    
    Moreover, with the convention $0\log 0=0$, since $0\leq \eta_j\leq 1$, we have $|\eta_j(x)\log \eta_j(x)|<1$ and $\eta_j(x)\log \eta_j(x)\rightarrow 0$ pointwise. By the dominated convergence theorem,
    \begin{equation*}
        \lim_{j\rightarrow\infty}\int_{\rn}f(x)\eta_j(x)\log\eta_j(x)dx=0.
    \end{equation*}
    Combined with \eqref{conv-cut1}, we obtain
    \begin{equation}\label{conv-log}
        \lim_{j\rightarrow\infty}\int_{\rn}f_j(x)\log f_j(x)dx=\int_{\rn}f(x)\log f(x)dx.
    \end{equation}

    Formula \eqref{conv-log} implies that $f_j\log f_j\in L^1(\rn)$ for sufficiently large $j$. Since $f_j\in W^{\beta,1}(\rn)$ is non-negative, non-zero and has compact support, Lemma \ref{C-compact support} implies that
    \begin{equation*}
        \bar\sigma_0+\frac{1}{n}\log\|f_j\|_1-\frac{1}{n\|f_j\|_1}\int_{\rn}f_j(x)\log f_j(x)dx\geq \frac{1}{n\omega_n}\int_{\sn}\log\rho_{\rR_0f_j}(u)du.
    \end{equation*}
    By $\|f_j\|_1\rightarrow\|f\|_1$, \eqref{conv-log} and Lemma \ref{cuf-off}, letting $j\rightarrow\infty$ gives
    \begin{equation}\label{C-step}
        \bar\sigma_0+\frac{1}{n}\log\|f\|_1-\frac{1}{n\|f\|_1}\int_{\rn}f(x)\log f(x)dx\geq \frac{1}{n\omega_n}\int_{\sn}\log\rho_{\rR_0f}(u)du.
    \end{equation}
    
Taking $\|f\|_1=1$ in \eqref{C-step}, multiplying by $n\omega_n$, and using
\eqref{log0-statement} together with polar coordinates, we obtain
        \[\sigma_0-\omega_n\int_{\rn}f(x)\log f(x)dx\ge \di\frac{\min\{f(x), f(y)\}-f(x)e^{-|x-y|} }{|x-y|^n}dxdy.\]
    
    Note that the finiteness of $\int f\log f$ ensures that $\log\rho_{\rR_0f}\in L^1(\sn)$. Therefore if equality holds in \eqref{C-step}, Theorem \ref{vlog-rri-thm} yields that
    \[\bar\sigma_0+\frac{1}{n}\log\|f^{\star}\|_1-\frac{1}{n\|f\|_1}\int_{\rn}f^{\star}(x)\log f^{\star}(x)dx=\tilde{V}_{\log}(B^n,\rR_0f^{\star})=\tilde{V}_{\log}(B^n,\rR_0f).\]
    Then Corollary \ref{ex-C1}  and Theorem \ref{vlog-rri-thm} give the equality characterization.
\end{proof}

\vskip 10pt

\noindent{\bf Acknowledgements:} 
This research was funded in whole or in part by the Austrian Science Fund (FWF) doi/10.55776/37030. For open access purposes, the authors have applied a CC BY public copyright license to any author accepted manuscript version arising from this submission.

\bibliographystyle{abbrv}
%\bibliography{new_bib}

\begin{thebibliography}{999}

%\bibitem{Pfi}
%R.E. Pfiefer, 
%\emph{Maximum and minimun sets for some geometric mean values}, J.Theoret. Probab., 3, 169-179, 1990.


\bibitem{AL}
F. J. Almgren, Jr. and E. H. Lieb,
\emph{Symmetric decreasing rearrangement is sometimes continuous},
J. Amer. Math. Soc. 2 (1989), 683-773.


\bibitem{ABG}
D. Alonso-Guti\'errez, J. Bernu\'es, B. Gonz\'alez Merino,
\emph{Zhang's inequality for log-concave functions},
Geometric aspects of functional analysis, {V}ol. {I}, Lecture Notes in Math., Springer, Cham (2020), 29-48.


%\bibitem{APM}
%L. Ambrosio, G. De Philippis, and L. Martinazzi, 
%\emph{Gamma-convergence of nonlocal
% perimeter functionals}, Manuscripta Math. 134 (2011), 377403, MR2765717, Zbl
% 1207.49051.

\bibitem{Aub}
T. Aubin,
\emph{Nonlinear Analysis on manifolds: Monge-Amp{\'e}re equations},
Springer, Berlin, 1982.


\bibitem{Aub2}
T. Aubin,
\emph{Probl\`emes isop\'erim\'etriques et espaces de Sobolev},
J. Differential Geom. 11 (1976), 573--598.

%\bibitem{BC}
%F. M. Ba$\rm \hat{e}$ta and X. Cai
%\emph{Affine chord Sobolev inequalities and radial mean bodies for functions},
%arXiv: 2511.12866.


\bibitem{Ball}
K. Ball,
\emph{Isometric problems in $\ell_p$ and sections of convex bodies},
PhD thesis, University of Cambridge, 1987.


%\bibitem{Be1}
%W. Beckner,
%\emph{Sharp Sobolev inequalities on the sphere and the Moser-Trudinger inequality},
%Ann. of Math. 138 (1993), 213-242

\bibitem{BP}
W. Beckner and M. Pearson,
\emph{On sharp Sobolev embedding and the logarithmic Sobolev inequality}, 
Bull. London Math Soc. 30 (1998), 80-84.

\bibitem{BH}
S. G. Bobkov and C. Houdr\'e,
\emph{Some connections between isoperimetric and Sobolev-type inequalities},
Mem. Amer. Math. Soc. 129 (1997), N. 616.


\bibitem{BobkovMadiman}
S. G. Bobkov and M. Madiman,
\emph{The entropy per coordinate of a random vector is highly constrained under convexity conditions},
IEEE Trans. Inform. Theory {57} (2011), 4940-4954.


%\bibitem{BBM}
%J. Bourgain, H. Brezis, and P. Mironescu, 
%\emph{Another look at Sobolev spaces},
% In: Optimal Control
% and Partial Differential Equations (J. L. Menaldi, E. Rofman and A. Sulem, eds.). A volume in honor of A. Bensoussans’s 60th birthday, Amsterdam: IOS Press; Tokyo: Ohmsha, 2001.

\bibitem{BBM}
J. Bourgain, H. Brezis, and P. Mironescu,
\emph{Another look at Sobolev spaces},
In: \emph{Optimal Control and Partial Differential Equations},
J. L. Menaldi, E. Rofman, and A. Sulem (eds.),
IOS Press, 2001.



%\bibitem{BBM2}
%J. Bourgain, H. Brezis, and P. Mironescu, \emph{Limiting embedding theorems for Wsp
% when $s\uparrow1$ and applications}, J. Anal. Math. 87 (2002), 77101, Dedicated to the memory of Thomas H. Wol , MR1945278, Zbl 1029.46030


\bibitem{BZ}
Y. D. Burago and V. A. Zalgaller,
\emph{Geometric inequalities},
Springer-Verlag, Berlin, 1988.


\bibitem{Bu}
A. Burchard, 
\emph{Cases of equality in the Riesz rearrangement inequality}, 
Ann. of Math. (2) 143 (1996), 499-527.



%\bibitem{CRS}
% L. Caffarelli, J.-M. Roquejo re, and O. Savin, 
% \emph{Nonlocal minimal surfaces}, 
% Comm. Pure Appl. Math. 63 (2010), 11111144, MR2675483, Zbl 1248.53009.

%\bibitem{CV}
%L. Caffarelli and E. Valdinoci, 
%\emph{Uniform estimates and limiting arguments for nonlocal minimal surfaces}, Calc. Var. Partial Di erential Equations 41 (2011), 203240,
% MR2782803, Zbl 05884582.

\bibitem{Cai}
X. Cai,
\emph{Affine logarithmic HLS and Beckner-Type logarithmic Sobolev inequalities},
arXiv: 2504.09251 (2025).

\bibitem{Cai2}
X. Cai,
\emph{Anisotropic fractional area measures},
arXiv: 2510.05279 (2025).


\bibitem{Ca}
E. Carlen,
\emph{Superadditivity of Fisher's information and logarithmic Sobolev inequalities},
J. Funct. Anal. 101 (1991), 194-211.

\bibitem{Ca2}
E. Carlen,
\emph{Duality and stability for functional inequalities}, Ann. Fac. Sci. Toulouse Math. (6) 26 (2017), 319-350.


\bibitem{Ci}
A. Cianchi,
\emph{A quantitative Sobolev inequality in BV},
J. Funct. Anal. 237 (2006), 466-481.

\bibitem{CF}
A. Colesanti and I. Fragal\`a,  
\emph{The first variation of the total mass of log-concave functions
and related inequalities},
Adv. Math. 244 (2013), 708–749.


\bibitem{Davy}
P. Davy,
\emph{Inequalities for moments of secant length},
Z. Wahrscheinlichkeitstheorie Verw. Geb. 68 (1984), 243-246.


\bibitem{DEFFL}
J. Dolbeault, M. J. Esteban, A. Figalli, R. L. Frank and M. Loss,
\emph{Sharp stability for Sobolev and log-Sobolev inequalities, with optimal dimensional dependence},
Camb. J. Math. 13 (2025), 359–430.

\bibitem{EG}
L. C. Evans and R. F. Gariepy,
\emph{Measure theory and fine properties of functions},
revised ed, CRC Press, Boca Raton, FL, 2015.


\bibitem{Fed}
H. Federer, 
\emph{Geometric measure theory}, 
Springer, Berlin, 1969. 

\bibitem{FF}
H. Federer and W. Fleming, 
\emph{Normal and integral currents}, 
Ann. of Math. (2) 72 (1960), 458-520.

\bibitem{FFMMM}
A. Figalli, N. Fusco, F. Maggi, V. Millot, and M. Morini,
\emph{Isoperimetry and stability properties of balls with respect to nonlocal energies},
Comm. Math. Phys. 336 (2015), 441-507.


\bibitem{FN}
A. Figalli and R. Neumayer,
\emph{Gradient stability for the Sobolev inequality: the case $p\ge 2$},
J. Eur. Math. Soc. (JEMS) 21 (2019), 319–354.


\bibitem{FZ}
A. Figalli and Y. R. Y. Zhang,
\emph{Sharp gradient stability for the Sobolev inequality},
Duke Math. J. 171 (2022), 2407-2459.


\bibitem{FS}
R. Frank and R. Seiringer,
\emph{Non-linear ground state representations and sharp Hardy inequalities},
J. Funct. Anal. 255 (2008), 3407-3430.


\bibitem{Garo}
N. Garofalo,
\emph{On the best constant in the nonlocal isoperimetric inequality of Almgren and Lieb},
Atti Accad. Naz. Lincei Rend. Lincei Mat. Appl. 31 (2020), 465-470.


%\bibitem{Gar1}
%R. J. Gardner,
%\emph{The Brunn-Minkowski inequality: A survey with proofs},
%available at https://faculty.gardner.wwu.edu/gorizia12.pdf.

%\bibitem{GHWY}
%R. J. Gardner, D. Hug, W. Weil, and D. Ye,
%\emph{The dual Orlicz-Brunn-Minkowski theory}, J. Math. Anal. Appl. 430 (2015), 810-829.

\bibitem{GZ}
R. J. Gardner and G. Zhang,
\emph{Affine inequalities and radial mean bodies}, 
Amer. J. Math. 120 (1998), 505-528.


\bibitem{GXZ}
L. Guo, D. Xi and Y. Zhao,
\emph{The $L_p$ chord Minkowski problem in a critical interval},
Math. Ann. (2023).


\bibitem{HV}
E. Hebey and M. Vaugon,
\emph{The best constant problem in the Sobolev embedding theorem for complete Riemannian manifolds},
Duke Math. J. 79 (1995), 235-279.


%\bibitem{HP}
%D. Hoffman and J. Spruck,
%\emph{A Sobolev inequality for Riemannian submanifolds}

\bibitem{Ho}
R. Howard,
\emph{The sharp Sobolev inequality and the Banchoff-Pohl inequality on surfaces},
Proc. Amer. Math. Soc. 126 (1998), 2779-2787.

%\bibitem{JL}
%D. Jerison and J. Lee,
%\emph{Extremals for the Sobolev inequality on the Heisenberg group and the CR Yamabe problem},
%J. Amer. Math. Soc. 1 (1988), 1-13.


\bibitem{HL1}
J. Haddad and M. Ludwig,
\emph{Affine fractional Sobolev and isoperimetric inequalities},
J. Differential Geom. 129 (2025), no. 3, 695-724.



\bibitem{HL3}
J. Haddad and M. Ludwig,
\emph{Affine Hardy-Littlewood-Sobolev inequalities}, 
J. Eur. Math. Soc. (2025), to appear.

\bibitem{HHLW}
J. Hu, Y. Huang, J. Lu and S. Wang,
\emph{The chord Gauss curvature flow and its $L_p$ chord Minkowski problem},
Acta Math. Sci. Ser. B (Engl. Ed.) 45 (2025), no. 1, 161–179.

%\bibitem{Kl}
%S. Klainerman,
%\emph{Remarks on the global Sobolev inequalities in the Minkowski space $\tr^{n+1}$},
%Comm. Pure Appl. Math. 40 (1987), 111-117.


%\bibitem{Kr}
%A. Kreuml,
%\emph{The anisotropic fractional isoperimetric problem with respect to unconditional unit balls}, 
%Comm. Pure Appl. Math. (7) 20 (2021), 783-799.

\bibitem{LSU}
D. Langharst, F. Mar\'in Sola and J. Ulivelli,
\emph{Higher-Order Reverse Isoperimetric Inequalities for Log-concave Functions},
arXiv:2403.05712 (2024).

\bibitem{Le}
M. Ledoux, 
\emph{Analytic and Geometric Logarithmic Sobolev Inequalities}, 
Journ\'ees \'equations aux d\'e riv\'ees partielles, Groupement de recherche 2434 du CNRS (2011), 1-15.

\bibitem{Leoni}
G. Leoni,
\emph{A First Course in Fractional Sobolev Spaces},
American Mathematical Society, 2023.

%\bibitem{Lieb2}
%E. Lieb,
%\emph{Existence and uniqueness of the minimizing solution of Choquard's nonlinear equation},
%Stud. Appl. Math. 57 (1977), 97-105.

\bibitem{Lieb}
E. Lieb, 
\emph{Sharp constants in the Hardy-Littlewood-Sobolev and related inequalities}, 
Ann. of Math. (2) 118 (1983), 349-374.

 


\bibitem{LL}
E. Lieb and M. Loss,
\emph{Analysis}, Second ed., Graduate Studies in Mathematics, vol. 14,  American Math. Soc., 2001.

\bibitem{Lud1}
M. Ludwig,
\emph{Anisotropic fractional perimeters}, 
J. Differential Geom. 96 (2014), 77-93.

\bibitem{Lud2}
M. Ludwig,
\emph{Anisotropic fractional Sobolev norms}, 
Adv. Math. 252 (2014), 150-157.

\bibitem{Lut1}
E. Lutwak,
\emph{Dual mixed volumes}, 
Pacific J. Math. 58 (1975), 531-538.

%\bibitem{Lut2}
%E. Lutwak,
%\emph{Inequalities for Hadwiger's harmonic quermassintegrals},
%Math. Ann. 280 (1988), 165-175.


\bibitem{LXYZ2020}
E. Lutwak, D. Xi, D. Yang and G. Zhang,
\emph{Chord measures in integral geometry and their Minkowski problems},
Comm. Pure Appl. Math.  77 (2024), 3277-3330.

\bibitem{Ma}
V. Maz$'$ya,
\emph{Lectures on isoperimetric and isocapacitary inequalities in the theory of Sobolev spaces},
Contemp. Math. 338 (2003), 307–340 

\bibitem{MS}
V. Maz$'$ya and T. Shaposhnikova,
\emph{On the Bourgain, Brezis, and Mironescu theorem concerning limiting embeddings of fractional Sobolev spaces}, J. Funct. Anal. 195 (2002), 230-238.

%\bibitem{Maggi}
%F. Maggi, 
%\emph{Sets of finite perimeter and geometric variational problems}, 
%Cambridge University Press, 2012.

\bibitem{Qin}
L. Qin,
\emph{Nonlocal energies of convex body and their log-Minkowski problem},
Adv. Math. 427, 2023: Art 109132.

\bibitem{Ren2}
D. Ren,
\emph{Topics in integral geometry},
World Scientific, Singapore, 1994.


\bibitem{Ren1}
D. Ren,
\emph{Two topics in integral geometry},
Proceedings of the 1981 Symposium on Differential Geometry and Differential Equations
(Shanghai–Hefei), Science Press, Beijing, 1984, 309-333.


\bibitem{Ro}
O. S. Rothaus,
\emph{Analytic inequalities, isoperimetric inequalities and logarithmic Sobolev inequalities},
J. Funct. Anal. 64 (1985), 296-313.


\bibitem{San}
L. A. Santal\'o,
\emph{Integral geometry and geometric probability},
Addison-Wesley, Reading, MA, 1976.


\bibitem{Sch}
R. Schneider,
\emph{Convex Bodies: the Brunn-Minkowski theory},
 Second expanded ed., Encyclopedia of Mathematics and its Applications, Cambridge
 University Press, 2014.


\bibitem{Sch2}
 R. Schneider,
 \emph{Inequalities for random flats meeting a convex body},
 J. Appl Prob. 22 (1985), 710-716.


 \bibitem{SW}
 R. Schneider and W. Weil,
 \emph{Stochastic and Integral Geometry},
 Probability and its Applications (New York), Springer-Verlag, Berlin, 2008.


\bibitem{Vis}
A. Visintin,
\emph{Nonconvex functionals related to multiphase systems}, 
SIAM J. Math. Anal. 21 (1990), 1281-1304.


\bibitem{XYZZ}
D. Xi, D. Yang, G. Zhang and Y. Zhao,
\emph{The $L_p$ chord Minkowski problem},
Adv. Nonlinear Stud. 23 (2023),  20220041.


\bibitem{XZ}
D. Xi and Y. Zhao,
\emph{Fractional affine area measures}, preprint (2025).


\bibitem{XC}
G. Xiong and W. Cheung,
\emph{Chord power integrals and radial mean bodies},
J. Math. Anal. Appl. 342 (2008), 629-637.



\bibitem{XS}
G. Xiong and X. Song,
\emph{Inequalities for chord power integrals},
J. Korean Math Soc. 45 (2008), 587-596.


\bibitem{Xu}
W. Xu, 
\emph{Entropy of chord distribution of convex bodies},
Proc. Amer. Math. Soc. 147 (2019), 3131-3141.


\bibitem{Yau}
S. T. Yau,
\emph{Sobolev inequality for measure spaces},
Tsing Hua lectures on geometry and analysis (Hsinchu, 1990-1991), 299-313, Internat. Press, Cambridge, MA, 1997.


%\bibitem{Zh1}
%G. Zhang
%\emph{Restricted chord projection and affine inequalities},
%Geom. Dedicata 39 (1991), 213-222.

\bibitem{Zh2}
G. Zhang,
\emph{Integral geometric inequalities},
Acta. Math. Sin. (Chin. Ser.) 34 (1991), 72-90.


\bibitem{Zh3}
G. Zhang,
\emph{Isoperimetric inequalities for integral geometric invariants of random lines},
Acta Math. Sci. Ser. B (Engl. Ed.) 45 (2025),  189-199.


\bibitem{Zh4}
G. Zhang,
\emph{The affine Sobolev inequality}, J. Differential Geom. 53 (1999), 183-202.



\end{thebibliography}

\end{document}